\documentclass[10pt]{amsart}

\usepackage{amssymb,amsmath,amscd,amsthm,amsfonts,wasysym,mathrsfs,amsxtra}
\input xy
\xyoption{all}

\newcommand{\RR}{\mathbb{R}}
\newcommand{\OO}{\mathscr{O}}

\newcommand{\MM}{\mathcal{M}}

\newcommand{\Cfield}{\mathbb{C}}

\newcommand{\Spec}{\textnormal{Spec}\,}

\newcommand{\image}{\textnormal{im}\,}

\newcommand{\HHom}{\mathscr{H}om} 
\newcommand{\Hom}{\textnormal{Hom}}
\newcommand{\dimension}{\textnormal{dim}\,}
\newcommand{\codimension}{\textnormal{codim}\,}

\newcommand{\Ext}{\textnormal{Ext}}
\newcommand{\EExt}{\mathscr{E}xt}

\newcommand{\al}{\alpha}

\newcommand{\Coh}{\textnormal{Coh}}

\newcommand{\Ap}{\mathcal{A}^p}

\newcommand{\Lo}{\overset{L}{\otimes}}

\newcommand{\arinj}{\ar@{^{(}->}}
\newcommand{\arsurj}{\ar@{->>}}

\newcommand{\Bob}{\mathcal B_{\omega, B}}

\newtheorem{theorem}{Theorem}[section]

\newtheorem{lemma}[theorem]{Lemma}
\newtheorem{pro}[theorem]{Proposition}
\newtheorem{proposition}[theorem]{Proposition}

\newtheorem{corollary}[theorem]{Corollary}
\theoremstyle{definition}
\newtheorem{definition}[theorem]{Definition}

\theoremstyle{remark}
\newtheorem{remark}[theorem]{Remark}

\numberwithin{equation}{section}

\begin{document}

\title[Stability and Fourier-Mukai Transforms]
{Stability and Fourier-Mukai Transforms on Higher Dimensional Elliptic Fibrations}

\author[Wu-yen Chuang, Jason Lo]{Wu-yen Chuang, Jason Lo}
\address{Department of Mathematics, National Taiwan University, Taipei, Taiwan} \email{wychuang@ntu.edu.tw}
\address{Department of Mathematics, University of Illinois at Urbana-Champaign, Urbana IL 61801, USA} \email{jccl@illinois.edu}


\begin{abstract}
We consider elliptic fibrations with arbitrary base dimensions, and generalise most of the results in \cite{Lo1, Lo2, Lo5}.  In particular, we check universal closedness for the moduli of semistable objects with respect to a polynomial stability that reduces to PT-stability on threefolds.  We also show openness of this polynomial stability.  On the other hand, we write down a criterion under which certain 2-term polynomial semistable complexes are mapped to torsion-free semistable sheaves under a Fourier-Mukai transform.  As an application, we construct an open immersion from a moduli of  complexes to a moduli of Gieseker stable sheaves on higher dimensional elliptic fibrations.
\end{abstract}


\maketitle

\section{Introduction}
Since its first introduction, Fourier-Mukai transforms have been proved to provide a useful method to study moduli problems on a variety $X$ in terms of moduli on the Fourier-Mukai partner $Y$.  For example, Bridgeland \cite{FMTes} showed that if $X$ is a relatively minimal elliptic surface, then Hilbert schemes of points on $Y$ are birationally equivalent to moduli of stable torsion-free sheaves on $X$.  If $X$ is an elliptic threefold, then Bridgeland-Maciocia \cite{BMef} showed that any connected component of a complete moduli of rank-one torsion-free sheaves is isomorphic to a component of the moduli of stable torsion-free sheaves on $Y$. We will mention only some works in this direction, and refer the readers to \cite{FMNT} for more details and a more comprehensive survey.

Since Bridgeland's introduction of stability conditions on triangulated categories \cite{StabTC}, there have been interests in understanding stable objects in the bounded derived category of coherent sheaves $D(X)$ of a variety $X$ and their moduli spaces. Using Fourier-Mukai transforms, it is possible to transform certain moduli problems for complexes on $X$ to moduli problems for sheaves on $Y$.  Recent related works along this direction include: Bernardara-Hein \cite{BH} and Hein-Ploog \cite{HP} for  elliptic K3 surfaces, Maciocia-Meachan \cite{MM} for rank-one Bridgeland stable complexes on Abelian surfaces,  Minamide-Yanagida-Yoshioka \cite{MYY,MYY2} for Bridgeland stable complexes on Abelian and K3 surfaces, the second author for K3 surfaces \cite{Lo4} and elliptic threefolds \cite{Lo5}.

\subsection{Overview of results}

In this paper, we consider elliptic fibrations $\pi : X \to S$ where the dimension of the base $S$ is at least two, together with a dual fibration $\widehat{\pi} : Y \to S$.  We generalise most of the results in \cite{Lo5}, where the dimension of the base was exactly two.  For many of the results in \cite{Lo5}, their proofs carry over to the higher dimensional case without any change; we restate these results in Section \ref{complexes}.  For some of the other results in \cite{Lo5}, however, we need to modify their proofs in major ways in order to prove them in higher dimensions.  The first such result is Theorem \ref{main1}, which roughly says that, if $F$ is a reflexive WIT$_1$ sheaf on an elliptic fibration, then it satisfies the vanishing condition
\[
\Ext_{D(X)}^1 (\mathcal B_X \cap W_{0,X}, F)=0.
\]

The threefold version of this theorem appeared as \cite[Theorem 2.19]{Lo5}.  To prove Theorem \ref{main1} for arbitrary base dimension, we need Lemma \ref{lemma13}, a result  on the codimensions of the sheaves $\EExt^p(E,A)$, for any reflexive sheaf $E$ and any coherent sheaf $A$  on $X$; this lemma is proved using a spectral sequence.  Theorem \ref{main1} allows us to identify the type of 2-term complexes $E$ that are mapped to torsion-free sheaves (in particular, we need $H^{-1}(E)$ to be torsion-free and reflexive).

Back in the case of elliptic threefolds in \cite{Lo5}, we considered complexes that were both $\sigma$-semistable and $\tilde{\sigma}$-semistable, where $\sigma$ was a polynomial stability of type `V2', and $\tilde{\sigma}$, being the dual stability of $\sigma$, was a polynomial stability of type `V3'.   In Section \ref{moduli1}, we consider polynomial stability conditions on higher dimensional varieties, particularly two classes which we call type W1 and type W2.  Stabilities of type W1 generalise the stabilities of type V2 on threefolds  from \cite{Lo3}, and include PT-stability (studied in \cite{Lo1, Lo2}); on the other hand, stabilities of type W2 generalise those of type V3 on threefolds from \cite{Lo3}.  We push  most of the results in \cite{Lo1,Lo2} to higher dimensions, including universal closedness for the moduli stack of PT-semistable objects, which is stated here as Theorem \ref{theorem2}.  Theorem \ref{theorem2} implies openness of semistability of type W1, which is stated as Corollary \ref{coro2}.  Having openness  allows us to speak of moduli stacks of polynomial semistable complexes.

%

In Section \ref{subsection-reflexive}, we study the condition of $H^{-1}(E)$ being torsion-free and reflexive when $E$ is a 2-term complex with cohomology sitting at degrees $-1$ and $0$.  We show that, when $\sigma$ is a polynomial stability of type W1, this condition is an open property for flat families of $\sigma$-semistable complexes.  As a consequence, we construct an open immersion from a moduli stack of polynomial stable complexes on $X$ to a moduli stack of stable sheaves on $Y$ in Theorem \ref{thm_a}.  And, as a byproduct of the machinery we develop in Section \ref{subsection-reflexive}, we show that  objects in the category $\mathfrak{D}$ described in \cite[Section 7.2]{BMTI} form moduli stacks, whether they are of types (a), (b) or (c).  In Theorem \ref{theorem5}, we show that a particular class of objects in $\mathfrak D$ of type (c) occur as the stable objects with respect to a polynomial stability.

In Theorem \ref{thm1} of Section \ref{moduli2}, we construct an equivalence of categories between a category $\mathcal{C}_X$ of 2-term complexes on $X$ and a category $\mathcal{C}_Y'$ of torsion-free sheaves on $Y$.  This theorem describes the objects in $D(X)$ and $D(Y)$ that we need to add in order to turn the aforementioned open immersion of stacks into an isomorphism of stacks.

Finally, in Section \ref{pure}, we consider torsion-free sheaves on $X$ that are taken to codimension-1 sheaves on $Y$ under Fourier-Mukai transformations.   Again, we generalise the threefold result \cite[Corollary 5.9]{Lo5}, so that we have an equivalence between the category of line bundles of fibre degree 0 on $X$, and the category of line bundles supported on sections of $\widehat{\pi}$.  These results resemble some of the results obtained using the spectral approach due to Friedman-Morgan-Witten, but do not make use of Fitting ideals.

We note that, although the conditions we impose on the Fourier-Mukai transforms we consider (properties (i) through (vi) in Section \ref{section-setup1}) may seem artificial, they are all satisfied by the elliptic threefolds considered in \cite[Section 9]{BMef}, and also by the Weierstrass fibrations of any dimension considered in \cite{FMNT}.

\medskip\medskip

{\bf Acknowledgements.} We would like to thank Jungkai Chen for raising a question leading to this paper. We also thank the referee for the thorough review and useful comments and suggestions. W.Y.C. was supported by NSC grant 101-2628-M-002-003-MY4 and a fellowship from the Kenda Foundation.  J.L.  Would like to thank the Taida Institute for Mathematical Sciences and the National Center for Theoretical Sciences (North) for the hospitality throughout his visits, during which part of this work was completed.

\section{Preliminaries and notation}

\subsection{Notation}

For a smooth projective variety $X$, we will always write $D(X)$ to denote its bounded derived category.  Given a t-structure on $D(X)$ with $\mathcal A$ as the heart, we will write $D^{\leq 0}_{\mathcal A} (X)$ (resp.\ $D^{\geq 0}_{\mathcal A}(X)$) to denote the subcategory of $D(X)$ consisting of complexes $E$ such that $\mathcal H^i_{\mathcal A}(E) = 0$ for all $i>0$ (resp.\ for all $i<0$), where $\mathcal H^i_{\mathcal A}(-)$ denotes the $i$-th cohomology functor with respect to the aforementioned t-structure. The cohomology functor $\mathcal H^i_{\Coh(X) }(-)$ with respect to the standard t-structure will be simply denoted by $H^i(-)$.

\subsection{The setup}\label{section-setup1}

Let us fix the following setting for the rest of the article.  We will assume that  $\pi : X \to S$ is a morphism satisfying:
 \begin{itemize}
 \item[(i)] $\pi$ is projective and flat;
 \item[(ii)] $X, S$ are smooth projective varieties;
  \end{itemize}
We will also assume that there exists another fibration $\widehat{\pi} : Y \to S$ (which might be isomorphic to $\pi$)  such that:
\begin{itemize}
\item[(iii)] the fibration $\widehat{\pi}$ also satisfies properties (i) and (ii);
\item[(iv)] $Y$ is a fine, relative moduli of stable sheaves on the fibres of $X$, while $X$ itself is also a fine, relative moduli of stable sheaves on the fibres of $Y$, and $\dimension X = \dimension Y$;
\item[(v)] the universal families from (iv) give us a pair of Fourier-Mukai transforms $\Psi : D(X) \to D(Y)$ and $\Phi : D(Y) \to D(X)$ such that $\Phi \Psi = \text{id}_{D(X)}[-1]$ and $\Psi \Phi = \text{id}_{D(Y)}[-1]$.
\end{itemize}

As in \cite{Lo5}, we introduce the following notations: we write $f$ to denote the Chern character of the structure sheaf of a smooth fibre of $\pi$, i.e.\ the `fibre class' of $\pi$.  Then for any object $E \in D(X)$, we define the fibre degree of $E$ to be
\[
d(E) = c_1(E)\cdot f,
 \]
 which is the degree of the restriction of $E$ to the generic fibre of $\pi$.  For the rest of this article, for any coherent sheaf $E$, we write $r(E)$ to denote its rank, and when $r(E)>0$, we  define
 \[
 \mu (E) = d(E)/r(E),
 \]
 which is the  slope of the restriction of  $E$ to the generic fibre.

We further assume:
\begin{itemize}
\item[(vi)] For any $E \in D(Y)$, we have
\begin{equation}\label{eqn-rdunderPhi}
 \begin{pmatrix} r(\Phi E) \\ d(\Phi E) \end{pmatrix} =
  \begin{pmatrix} c & a \\ d & b \end{pmatrix} \begin{pmatrix} r(E) \\ d(E) \end{pmatrix}
\end{equation}
for some element
 \[
  \begin{pmatrix} c & a \\ d & b \end{pmatrix} \in \text{SL}_2(\mathbb{Z})
 \]
where $a>0$.  Therefore, $Y$ is a relative moduli of stable sheaves of rank $a$ and degree $b$ on fibres of $\pi$.
\end{itemize}
As a result of assumption (vi), we also have, for any $E \in D(X)$,
\begin{equation}\label{eqn-rdunderPsi}
\begin{pmatrix} r(\Psi E) \\ d(\Psi E) \end{pmatrix} =
  \begin{pmatrix} -b & a \\ d & -c \end{pmatrix} \begin{pmatrix} r(E) \\ d(E) \end{pmatrix}.
\end{equation}
And hence (taking into account assumption (iv)) $X$ is a relative moduli of stable sheaves of rank $a$ and degree $-c$ on fibres of $\widehat{\pi}$.

\begin{remark}
The elliptic surfaces studied by Bridgeland \cite{FMTes} and the elliptic threefolds studied by Bridgeland-Maciocia \cite[Section 8]{BMef} all possess properties (i) through (vi) above.
\end{remark}

For any complex $E \in D(X)$, we write $\Psi^i(E)$ to denote  $H^i(\Psi (E))$, i.e.\ the cohomology of $\Psi (E)$ with respect to the standard t-structure on $D(Y)$; if $E$ is a sheaf sitting at degree 0, we have that $\Psi^i (E)=0$ unless $0\leq i \leq 1$, i.e.\ $\Psi (E) \in D^{[0,1]}_{\Coh (Y)}(Y)$.  The same statements hold for $\Phi$ and $Y$. A complex $E$ is called $\Psi$-WIT$_i$ if $\Psi(E)=\widehat{E}[-i]$ for some coherent sheaf $\widehat{E}$ on $Y$.

We also define the following full subcategories of $\Coh (X)$, all of which are extension-closed:
\begin{align*}
 \mathcal T_X &= \{ \text{torsion sheaves on } X \} \\
  \mathcal F_X &= \{ \text{torsion-free sheaves on }X \} \\
  W_{0,X} &= \{ \Psi\text{-WIT}_0 \text{ sheaves on }X \}  \\
  W_{1,X} &= \{ \Psi\text{-WIT}_1 \text{ sheaves on }X \} \\
  \mathcal B_X &= \{ E \in \Coh (X): r(E)=d(E)=0 \} \\
    \Coh (X)_{r>0} &= \{ E \in \Coh (X) : r(E)>0\}.
\end{align*}
And for any $s \in \mathbb{R}$, we define
\begin{align*}
  \Coh (X)_{\mu > s} &= \{ E \in \Coh (X)_{r>0}: \mu (E) >s \} \\
  \Coh (X)_{\mu=s} &= \{ E \in \Coh (X)_{r>0}: \mu (E) =s \} \\
  \Coh (X)_{\mu<s} &= \{ E \in \Coh (X)_{r>0}: \mu (E) <s \}.
\end{align*}
We  define the corresponding full subcategories of $\Coh (Y)$ similarly.

For any nonnegative integer $i \leq \dimension (X)$, we write $\Coh_{\leq i}(X)$ to denote the subcategory of $\Coh (X)$ consisting of coherent sheaves supported in dimension $i$ or lower, and write $\Coh_{\geq i}(X)$ to denote the subcategory of $\Coh (X)$ consisting of coherent sheaves without subsheaves supported in dimension at most $i-1$.  For integers $0 \leq d' < d \leq \dimension (X)$, the category $\Coh_{\leq d'}(X)$ is a Serre subcategory of $\Coh_{\leq d}(X)$, and so we can form the quotient category $\Coh_{d,d'}(X) := \Coh_{\leq d}(X)/\Coh_{\leq d'}(X)$.  For objects $F$ in $\Coh_{d,d'}(X)$, we write $p_{d,d'}(F)$ to denote the reduced Hilbert polynomial of $F$, modulo polynomials over $\mathbb{Q}$ of degree at most $d'-1$.

\section{COMPLEXES AND FOURIER-MUKAI TRANSFORMS}
\label{complexes}

In this section, we collect many technical results on Fourier-Mukai transforms between $D(X)$ and $D(Y)$, which will be used to relate moduli stacks on $X$ and $Y$.

All the lemmas and theorems in this section except Lemma \ref{lemma13} have appeared in \cite[Section 2.4]{Lo5} before, where they were proved for the case of $X$ being a threefold (i.e.\ when the base $S$ is of dimension two).   For Lemma \ref{lemma1} through Theorem \ref{thm0}, all their proofs in the threefold case in \cite{Lo5} generalise in a straightforward manner to higher dimensions, and so we refer the readers to \cite{Lo5} for their proofs.  Lemma \ref{lemma13} is the new technical result we need for higher dimensions; it is an integral part of the proof of Theorem \ref{main1}.

\begin{lemma}\cite[Lemma 2.2]{Lo5}.\label{lemma1} 
If we define
\[
\mathcal B^\circ_X := \{ E \in\Coh (X): \Hom (\mathcal B_X, E)=0\},
 \]
 then $(\mathcal B_X, \mathcal B_X^\circ)$ is a torsion pair in $\Coh (X)$.
\end{lemma}

\begin{lemma}\label{lemma2} 
\cite[Lemma 6.2]{FMTes}
Let $E$ be a sheaf of positive rank on $X$.  If $E$ is $\Psi$-WIT$_0$, then $\mu (E) \geq b/a$.  If $E$ is $\Psi$-WIT$_1$, then $\mu (E) \leq b/a$.
\end{lemma}

\begin{lemma}\cite[Lemma 2.6]{Lo5}\label{lemma3} 
If  $T$ is a $\Psi$-WIT$_1$ torsion sheaf on $X$, then $T \in \mathcal B_X$.
\end{lemma}

\begin{remark}\cite[Remark 2.7]{Lo5} 
\label{remark0}
Given any $E \in D^b(X)$, we have $r(\Psi E) = -b\cdot r(E)+a\cdot d(E)$.  So when $E$ has positive rank, $\mu (E)=b/a$ is equivalent to $r(\Psi E)=0$.  In other words, if $E$ is a $\Psi$-WIT$_1$ sheaf on $X$ of positive rank with $\mu (E)=b/a$, then $\widehat{E}$ is a torsion sheaf on $Y$.
\end{remark}

\begin{lemma}\cite[Lemma 2.8]{Lo5}\label{lemma4} 
Suppose $E$ is a $\Psi$-WIT$_0$ sheaf on $X$ and $r(E)>0$.  Then $\mu (E)>b/a$.
\end{lemma}

Lemma \ref{lemma4} is slightly stronger than the second part of Lemma \ref{lemma2}.

\begin{lemma}\cite[Lemma 2.9]{Lo5}\label{lemma5} 
Suppose $T \in \mathcal B_X$.   Then $\Psi^0(T), \Psi^1(T)$ are both torsion sheaves and lie in $\mathcal B_Y$.
\end{lemma}

\begin{lemma}\cite[Lemma 2.11]{Lo5}\label{lemma6} 
 Let $E$ be a nonzero $\Psi$-WIT$_0$ sheaf of any rank on $X$ such that  $E \in \mathcal B_X^\circ$.  Then $\widehat{E}$ is a nonzero torsion-free sheaf.
\end{lemma}

\begin{lemma}\cite[Lemma 2.10]{Lo5}\label{lemma7} 
We have an equivalence of categories
\begin{gather}\label{equiv1}
\mathcal F_X \cap \{ E \in \Coh (X) : \Ext^1 (\mathcal B_X \cap W_{0,X}, E)=0\} \cap W_{1,X} \overset{\Psi [1]}{\to} \mathcal B_Y^\circ \cap W_{0,Y}.
\end{gather}
\end{lemma}

In order to prove Lemma \ref{lemma7}, we need the following Lemma \ref{lemma8} and Lemma \ref{lemma9}:

\begin{lemma}\cite[Lemma 2.11]{Lo5}\label{lemma8}  
Let $F$ be a $\Phi$-WIT$_0$ sheaf on $Y$.  Then $\widehat{F}$ is a torsion-free sheaf on $X$ if and only if $\Hom (\mathcal B_Y \cap W_{0,Y},F)=0$.
\end{lemma}

\begin{lemma}\cite[Lemma 2.12]{Lo5}\label{lemma9}  
Let $F$ be a $\Phi$-WIT$_0$ sheaf on $Y$.  Then
\[
  \Hom (\mathcal B_Y \cap W_{1,Y}, F) \cong \Ext^1 (\mathcal B_X \cap W_{0,X},\widehat{F}).
\]
\end{lemma}

\begin{lemma} \label{lemma10} 
\cite[Lemma 6.4]{FMTes}
Let $E$ be a torsion-free sheaf on $X$ such that the restriction of $E$ to the general fibre of $\pi$ is stable.  Suppose $\mu (E) < b/a$.  Then $E$ is $\Psi$-WIT$_1$.
\end{lemma}

\begin{lemma}\cite[Lemma 2.14]{Lo5}\label{lemma11} 
The functor $\Psi [1]$  restricts to an equivalence of categories
\begin{gather}\label{equiv2}
W_{1,X} \cap \Coh (X)_{r >0} \cap \Coh (X)_{\mu < b/a} \overset{\Psi [1]}{\to} W_{0,Y} \cap \Coh (Y)_{r>0} \cap \Coh (Y)_{\mu > -c/a}.
\end{gather}
\end{lemma}

\begin{lemma}\cite[Lemma 2.15]{Lo5}\label{lemma12} 
The functor $\Psi [1]$ restricts to an equivalence of categories
\begin{gather}\label{equiv3}
  W_{1,X} \cap \Coh (X)_{r>0} \cap \Coh (X)_{\mu = b/a} \overset{\Psi [1]}{\to} W_{0,Y} \cap (\mathcal T_Y \setminus \mathcal B_Y).
\end{gather}
\end{lemma}

\begin{theorem}\cite[Theorem 2.17]{Lo5}\label{thm0} 
Suppose $F$ is a coherent sheaf on $X$ such that $F$ is torsion-free, $\Psi$-WIT$_1$
and $\widehat{F}$ restricts to a torsion-free sheaf on the generic fibre of $\widehat{\pi}$.
Then $\widehat{F}$ is a torsion-free sheaf if and only if
\begin{equation}\label{eq1}
\Ext^1 (\mathcal B_X \cap W_{0,X}, F)=0.
\end{equation}
\end{theorem}

\begin{lemma}\label{lemma13}
Let $X$ be a smooth projective variety, $E$ a reflexive sheaf on $X$, and $A$ any coherent sheaf on $X$.  Then $\codimension \EExt^q (E,A) \geq q+2$ for all $q >0$.
\end{lemma}

\begin{proof}
Consider the two derived functors $F, G : D(X) \to D(X)$ where $F (-) := R\HHom (-,\omega_X)$ and $G (-) := R\HHom (-,A)$.  Then for any complex $C$, we have $(G \circ F)(C) \cong C \Lo (A \otimes \omega_X^\ast)$. By \cite[Proposition 2.66]{FMTAG} (also see \cite[Lemma 1.1.8]{HL}), for any coherent sheaf $C$ on $X$, we have a spectral sequence
\begin{equation}
E_2^{p,q} := \EExt^p ( \EExt^{-q} (C,\omega_X), A) \Rightarrow H^{p-q}(C \Lo A')\ ,
\end{equation}
where $A' := A \otimes \omega_X^\ast$.

Since $E$ is reflexive, we have $E = \EExt^0 (C,\omega_X)$ for some coherent sheaf $C$ by \cite[Proposition 1.1.10]{HL}. As in the argument in \cite[p.6]{HL}, the term $E^{p,0}_2$ fits in the short exact sequences
\begin{equation*}
0 \to E^{p,0}_3 \to E_2^{p,0} \to E^{p+2,-1}_2
\end{equation*}
(since $E_2^{p,q}=0$ for $q>0$).  In fact, we have a short exact sequence
\begin{equation*}
0 \to E^{p,0}_{r+1} \to E^{p,0}_r \to E^{p+r,-(r-1)}_r \text{\quad for all $r\geq 2$}.
\end{equation*}
Since we also have $E_\infty^{p,0} = H^p (C \Lo A')=0$ for $p>0$, we have
\begin{align*}
  \dimension E_2^{p,0} &\leq \max{ \{ \dimension E_2^{p+2,-1}, \dimension E^{p,0}_3 \}}  \\
    &\leq \max{ \{ \dimension E_2^{p+2,-1}, \dimension E_3^{p+3,-2}, \dimension E^{p,0}_4\} } \\
    &\vdots \\
   &\leq \max_{r \geq 2} \{\dimension E_r^{p+r,-(r-1)}\}.
\end{align*}
So it suffices for us to show that $E_r^{p+r,-(r-1)}$, for $p\geq 0$ and $r \geq 2$, has codimension at least $p+r$.  It further suffices for us to show that for any coherent sheaves $E,F$ on $X$, we have $\codimension \EExt^p (E,F)  \geq p$ for any $p>0$.

Write $F_0 := F$.  For each integer $i \geq 0$, we take any surjection $\OO_X (m_i)^{\oplus r_i} \twoheadrightarrow F_i$ for some $m_i \ll 0$ and $r_i$, and let $F_{i+1}$ be the kernel.  Hence we have a short exact sequence
\begin{equation}\label{ses1}
  0 \to F_{i+1} \to \OO_X (m_i)^{\oplus r_i} \to F_i \to 0
\end{equation}
for any $i \geq 0$, where  $F_{i+1}$  is necessarily torsion-free.

Applying the functor $\EExt^p (E,-)$ to \eqref{ses1} when $i=0$, we obtain an exact sequence
\[
  \EExt^p (E,\OO_X (m_0))^{\oplus r_0} \to \EExt^p (E,F_0) \to \EExt^{p+1} (E,F_{1}).
\]
By \cite[Proposition 1.1.6(i)]{HL}, we have $\codimension \EExt^p (E,\OO_X (m_0)) \geq p$.  Hence it suffices to show $\codimension \EExt^{p+1} (E,F_1) \geq p$.

Applying the functor $\EExt^{p+1} (E,-)$ to \eqref{ses1} when $i=1$, we obtain an exact sequence
\[
  \EExt^{p+1} (E,\OO_X (m_1))^{\oplus r_1} \to \EExt^{p+1} (E,F_1) \to \EExt^{p+2} (E,F_2),
\]
where $\codimension \EExt^{p+1} (E,\OO_X (m_1)) \geq p+1$  by \cite[Proposition 1.1.6(i)]{HL} again.  Hence it suffices to show $\codimension \EExt^{p+2} (E,F_2) \geq p$, and so on.

Since $X$ is smooth of dimension $n$, the sheaf $E$ has  homological dimension at most $n$, and so $\EExt^{p+r} (E,F_r)=0$ whenever $p+r>n$.  Hence we are done.

\end{proof}

\begin{theorem}\label{main1} 
Suppose $\pi: X \to S$ is an elliptic fibration whose fibers are all Cohen-Macaulay curves with trivial dualising sheaves. If $F$ is a reflexive $\Psi$-WIT$_1$ sheaf on $X$, then $F$ satisfies
$
\Ext^1 (\mathcal B_X \cap W_{0,X}, F)=0.
$
\end{theorem}

\begin{proof}
We would like to show $\Ext^1(A, F)=0$ for any $A \in \mathcal B_X \cap W_{0,X}$. Using Serre duality, we have $\Ext^1(A,F)=\Ext^{n-1}(F,A\otimes \omega_X)$. Consider the local-to-global spectral sequence for $\Ext$,
\begin{equation}\label{eq2}
  E_2^{p,q} = H^p (X, \EExt^q (F,A\otimes \omega_X)) \Rightarrow \Ext^{p+q}(F,A\otimes \omega_X).
\end{equation}

Since $F$ is reflexive, by Lemma \ref{lemma13} we have $\codimension \EExt^q (E,A\otimes \omega_X) \geq q+2$ for all $q >0$. Therefore the only nonvanishing term in $E_2^{p,q}$ for $p+q=n-1$ is $E_2^{n-1,0}=H^{n-1}(X,\EExt^0 (F,A\otimes \omega_X))$ and we have a surjection
 \begin{equation}\label{eq3}
  H^{n-1} (X,\EExt^0 (F,A\otimes \omega_X)) \twoheadrightarrow \Ext^{n-1} (F,A\otimes \omega_X).
 \end{equation}

We can further assume the support of $\pi_{*}A$ is a reduced scheme, following \cite[Theorem 2.19, Step 2]{Lo5}. Let $C:=\text{supp}(\pi_{*}A)$ and the support of $A$ is contained in a subscheme $D$ which fits into the Cartesian diagram
\begin{equation}
\xymatrix{
 D \ar[d]^{\pi}\arinj[r]^\iota & X \ar[d]^\pi \\
 C \arinj[r] & S
}
\end{equation}

First note that we have
\[
H^{n-1}(X, \HHom (F, A \otimes \omega_X)) \cong H^{n-1}(D, \bar{A})
\] where $\bar{A}$ is a coherent sheaf on $D$ satisfying $\iota_{*}\bar{A} = \HHom(F, A \otimes \omega_X)$.
We apply the Leray spectral sequence to $\pi$ and obtain
\[
  E_2^{p,q} = H^p (C, R^q \pi_{*}(\bar{A})) \Rightarrow H^{p+q}(D, \bar{A}).
\]
Since all the fibres are 1-dimensional, the only nonvanishing terms in $E_2^{p,q}$ for $p+q=n-1$ are such that $(p,q)=(n-1,0),(n-2,1)$. Since $A \in \mathcal B_X$, the dimension of $D$ is at most $n-1$. If the dimension of $D$ is strictly less than $n-1$, then the dimension of $C$ is at most $n-3$. In this case there is nothing to prove. Hence it suffices to assume that $\text{dim}(D)=n-1$ and $\text{dim}(C)=n-2$. And we have $H^{n-2}(C, R^1 \pi_{\ast}(\bar{A})) \cong H^{n-1}(D, \bar{A})$. Now it suffices to show that the dimension of the support of $R^1 \pi_{*}(\bar{A})$ is at most $n-3$.
It is equivalent to showing that $R^1 \pi_\ast (\bar{A})$ has codimension at least 1 in $C$, i.e.
\begin{equation}
\text{for a general closed point } s \in C, \text{we have }
R^1 \pi_\ast (\bar{A}) \otimes k(s) =0.
\end{equation}

By generic flatness \cite[052B]{stacks}, $\bar{A}$ is flat over an open dense subscheme of $C$.  Now, let $s \in C$ be a general closed point, $g$ be the fibre $\pi^{-1}(s)$, and $\bar{A}|_s$ be the (underived) restriction of $\bar{A}$ to the fibre $g$ over $s$.  By cohomology and base change \cite[Theorem III 12.11]{Harts}, we have
\[
R^1 \pi_\ast (\bar{A}) \otimes k(s) \cong H^1 (g,\bar{A} |_s) .
\] So the theorem would be proved if we can show that $H^1(g,\bar{A}|_s)=0$.

By our assumptions, the fibre $g := \pi^{-1}(s)$ is a projective Cohen-Macaulay curve with trivial dualising sheaf. Using Serre duality, we have
\begin{equation}\label{eqSerre}
  H^1(g,\bar{A}|_s) \cong \Ext^1_g (\OO_g, \bar{A}|_s) \cong \Hom_g (\bar{A}|_s, \OO_g).
\end{equation}

Denote by $\Psi_s$ the induced Fourier-Mukai transform on the fibres $D(X_s) \to D(Y_s)$. Following \cite[Theorem 2.19, Step 4]{Lo5}, where \cite[Proposition A.85, (6.3), Proposition 6.1]{FMNT} are applied, we can similarly show that $A|_s$ is $\Psi_s$-WIT$_0$ for a general closed point $s \in C$.

Since $F$ is reflexive, it is locally free outside a $(n-3)$-dimensional closed subset $Z$ of $X$. So its locally free locus is still open and nonempty in $C$. Following \cite[Theorem 2.19, Step 5]{Lo5},
we can show $F|_s$ is $\Psi_s$-WIT$_1$ for a general closed point $s \in C$.

Now we have, for a general $s \in C$,
\begin{align*}
  \Hom_g (\bar{A}|_s, \OO_g) & = \Hom_g ( \HHom (F,A \otimes \omega_X)|_s, \OO_g) \\
  &\cong \Hom_g (A|_s, F|_s),
\end{align*} which must vanish  since $A|_s$ is $\Psi_s$-WIT$_0$ and $F|_s$ is $\Psi_s$-WIT$_1$. This completes the proof of the theorem.
\end{proof}

Theorem \ref{main1} combined with Lemma \ref{lemma10}, \ref{lemma11}, and Theorem \ref{thm0} gives the following:
\begin{corollary}\label{coro1}
Suppose $\pi: X \to S$ is an elliptic fibration whose fibres are all Cohen-Macaulay with trivial dualising sheaves. Then for any reflexive sheaf $F$ with  $\mu(F)<b/a$ such that its restriction to the generic fibre of $\pi$ is stable, we have $F$ is $\Psi$-WIT$_1$ and $\widehat{F}$ is torsion-free and stable with respect to some polarisation on $Y$.
\end{corollary}

\begin{proof}
Take any reflexive sheaf $F$ as described. Then $F$ is $\Psi$-WIT$_1$ due to Lemma \ref{lemma10} and we have $r(\widehat{F}) \neq 0$ by Lemma \ref{lemma11}.  Then $\widehat{F}$ is torsion-free by Theorem \ref{thm0} and Theorem \ref{main1}. By \cite[Lemma 9.5]{BMef} and \cite[Lemma 2.1]{BMef}, $\widehat{F}$ is stable on $Y$ with respect to some polarisation.
\end{proof}

\section{Moduli of Stable Complexes}
\label{moduli1}

In this section, we will construct an open immersion from a moduli of  2-term complexes on $X$ to a moduli space of Gieseker stable sheaves on $Y$.  Throughout this section, suppose $n \geq 3$ and consider the following heart of a t-structure
\[
\mathcal{A}^{p} = \langle {\rm Coh}_{\leq n-2}(X),
{\rm Coh}_{\geq n-1}(X)[1] \rangle.
\] The heart is obtained from ${\rm Coh}(X)$ by tilting once.

In the following, we make use of polynomial stability conditions on the derived category  $D^b(X)$ in the sense of  Bayer \cite{BayerPBSC}.   Included in Appendix \ref{PSC} are  some basics on polynomial stability conditions.

We consider two different types of  polynomial stability conditions, W1 and W2, on $X$.  For either of these types, we require that  no two of the stability vectors $\rho_i$ are collinear.  We impose the following additional assumptions:
\begin{itemize}
\item[For W1:] we have $\rho_0, \rho_1, \cdots, \rho_{n-2}, -\rho_{n-1}, -\rho_n \in \mathbb{H}$, as well as  $\phi(\rho_0) > \phi(-\rho_n)$,  $\phi(-\rho_{n-1}) > \phi(-\rho_n)$, and $\phi(-\rho_n)>\phi(\rho_i)$ for $1 \leq i \leq n-2$.
\item[For W2:] we have $\rho_0, \rho_1, \cdots, \rho_{n-2}, -\rho_{n-1}, -\rho_n \in \mathbb{H}$, as well as $\phi(-\rho_{n-1}) > \phi(-\rho_n)$, and $\phi(\rho_i) > \phi(-\rho_n)$ for $0 \leq i \leq n-2$.
\end{itemize}

Figures \ref{figure-W1} and \ref{figure-W2} below illustrate possible configurations of the $\rho_i$ for stabilities of types W1 and W2.  Note, for instance, that under our definition it is possible for a polynomial stabiliity of type W1 to have $\phi (\rho_0) > \phi (-\rho_{n-1})$.



\begin{figure*}[h]
\centering
\setlength{\unitlength}{1mm}
\begin{picture}(50,40)
\multiput(0,15)(1,0){50}{\line(1,0){0.5}}
\multiput(25,0)(0,1){40}{\line(0,1){0.5}}
\multiput(36,19)(0,1){3}{\line(0,1){0.5}}
\put(25,15){\vector(-4,1){15}}
\put(0,18.6){$-\rho_{n-1}$}
\put(25,15){\vector(-1,1){13}}
\put(7.5,28){$\rho_0$}
\put(25,15){\vector(1,2){10}}
\put(30,36){$-\rho_n$}
\put(25,15){\vector(3,1){14}}
\put(39.5,20){$\rho_{n-2}$}
\put(25,15){\vector(3,2){14}}
\put(39.5,24){$\rho_{1}$}
\end{picture}
\caption{A possible configuration of the $\rho_i$ for $W1$}
\label{figure-W1}
\end{figure*}
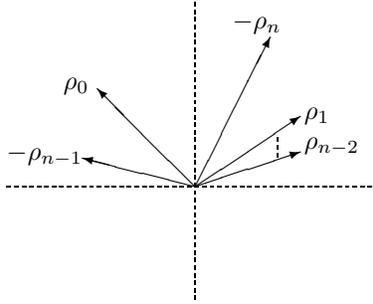

\begin{figure*}[h]
\centering
\setlength{\unitlength}{1mm}
\begin{picture}(50,40)
\multiput(0,15)(1,0){50}{\line(1,0){0.5}}
\multiput(25,0)(0,1){35}{\line(0,1){0.5}}
\multiput(22,29)(1,0){5}{\line(1,0){0.5}}
\put(25,15){\vector(-4,1){15}}
\put(0,18.6){$-\rho_{n-1}$}
\put(25,15){\vector(-1,1){16}}
\put(5,32){$\rho_0$}
\put(25,15){\vector(3,2){16}}
\put(40,27){$-\rho_n$}
\put(25,15){\vector(1,4){4}}
\put(25,15){\vector(-1,3){6}}

\put(15,30){$\rho_1$}
\put(30,30){$\rho_{n-2}$}
\end{picture}
\caption{A possible configuration of $\rho_i$ for $W2$}
\label{figure-W2}
\end{figure*}
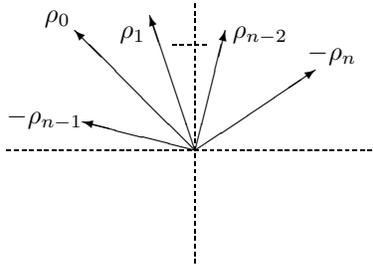

Using the terminology from \cite{Lo3}, when $X$ is of dimension three, stabilities of type W1 coincide with stabilities of type V2 (which includes PT stability, a stability that was studied in \cite{Lo1, Lo2}), while stabilities of type W2 coincide with stabilities of type V3.

The following Lemma \ref{lemma_a} is analogous to \cite[Proposition 2.24]{Lo2}, with essentially the same proof.

\begin{lemma}\label{lemma_a}
Let $\sigma$ be a polynomial stability condition of type  W1 and $E \in \mathcal{A}^p$ a 2-term complex with nonzero rank. Then conditions (1) through (3) below hold if $E$ is $\sigma$-semistable:
\begin{itemize}
\item[(1)] $H^{-1}(E)$ is a $\mu$-semistable torsion-free sheaf;
\item[(2)] $H^{0}(E)$ is 0-dimensional;
\item[(3)] $\Hom_{D(X)}(\mathcal{O}_x,E)=0$ for any $x \in X$, where $\mathcal{O}_x$ is the skyscraper sheaf at the closed point $x$.
\end{itemize}
When $ch_0(E)$ and $ch_1(E)$ are relatively prime, $E$ is $\sigma$-semistable if and only if (1) through (3) hold.
\end{lemma}

\begin{remark}\label{remark7}
As in the case of PT-semistability on threefolds, if $E \in \mathcal{A}^p$ is $\sigma$-semistable where $\sigma$ is of type W1, then $H^{-1}(E)$ is semistable in $\Coh_{n,n-2}(X)$ (see \cite[Section 3.1]{Lo2}).
\end{remark}

With the same proof as in \cite[Lemma 3.2]{Lo3}, we have:

\begin{lemma}\label{lemma_b}
Let $\tilde{\sigma}$ be a polynomial stability condition of type W2 and $E \in \mathcal{A}^p$ be a $\tilde{\sigma}$-semistable 2-term complex with nonzero rank. Then $H^{-1}(E)$ is a $\mu$-semistable reflexive sheaf.
\end{lemma}

\begin{remark}\label{remark14}
We do not make any significant use of polynomial stabilities of type W2 in this article.  Suppose $\sigma$ is a polynomial stability of type W1.  In the case of threefolds as in \cite{Lo5}, in order for $\Psi$  to take a $\sigma$-semistable complex $E \in D(X)$ to a stable sheaf, we assumed additionally that $E$ is  $\tilde{\sigma}$-semistable with $\tilde{\sigma}$ of type W2, and that $E$ satisfies property (P) (see Section \ref{subsection-open-immersion} below).  In this article, however, we find that the additional requirement of $E$ being $\tilde{\sigma}$-semistable can be replaced by the more general condition of $H^{-1}(E)$ being reflexive - see Section \ref{subsection-comparison}.
\end{remark}

\subsection{Openness of stabilities of types W1 and W2}


 When $\sigma$ is a polynomial stability, we  want to speak of  moduli stacks of $\sigma$-semistable  objects.  In order for these moduli stacks to exist, we need to show that being $\sigma$-semistable is an open property for flat families of complexes.  We do this for polynomial stabilities of type W1 below.

To begin with, by Lemma \ref{lemma_a} and Remark \ref{remark7}, we have the following analogue of \cite[Proposition 3.1]{Lo2}, with essentially the same proof:

\begin{pro}\label{prop5}
For flat families of objects in $\mathcal{A}^p$ of nonzero rank, properties (1), (2) and (3) in Lemma \ref{lemma_a} together form an open condition.
\end{pro}

The proof of  \cite[Lemma 3.2]{Lo2} also carries over to the case of stabilities of type W1, giving us:

\begin{lemma}\label{lemma23}
Fix an $ch_0 >0$.  Let $\sigma$ be a polynomial stability on $D(X)$ of type W1.  For any $ch_1, ch_2, \cdots, ch_n$, define the set of injections in $\mathcal{A}^p$
\begin{multline*}
  \mathcal{S} := \{ E_0 \hookrightarrow E : E_0 \text{ is a maximal destabilising subobject of $E$ in $\mathcal{A}^p$ w.r.t.\ $\sigma$},  \\
    \text{ where $E$ has properties (1), (2) and (3) and $ch(E)=ch$} \}.
\end{multline*}
Then the set
\[
  \mathcal{S}_{sub} := \{E_0 : E_0 \hookrightarrow E \text{ is in }\mathcal{S}\}.
\]
is bounded.
\end{lemma}

To be precise, we list here how the results in \cite{Lo1} generalise to stabilities of type W1 in higher dimensions:

\begin{lemma}\cite[Lemma 3.2]{Lo1}\label{lemma24}
Let $E \in \mathcal{A}^p$ be an object of rank zero, and $\sigma$ be a polynomial stability of type W1. Suppose $E$ is of dimension $n-1$ and $E$ is $\sigma$-semistable; then:
\begin{enumerate}
\item[(a)] if $\phi (\rho_0)>\phi (-\rho_{n-1})$, then $H^0(E)$ must be 0-dimensional;
\item[(b)] if $\phi (\rho_0)< \phi (-\rho_{n-1})$, then $E=H^{-1}(E)[1]$.
\end{enumerate}
If $E$ is of dimension at most $n-2$, then $E$ is $\sigma$-semistable iff $E=H^0(E)$ is a Gieseker semistable sheaf.
\end{lemma}
Note that, in case (a) above, we do not necessarily know that $H^{-1}(E)$ is a Gieseker semistable sheaf.  This is different from the case of PT stability on threefolds.

\begin{lemma}\cite[Proposition 3.4]{Lo1}\label{lemma25}
Let $\sigma$ be a polynomial stability of type W1, and $ch$ a fixed Chern character where $ch_0 \neq 0$.  Then the set of $\sigma$-semistable obejcts $E \in \Ap$ of Chern character $ch$ is bounded.
\end{lemma}

For the next proposition, we write $k$ for the ground field of the variety $X$, $R$ for an arbitrary discrete valuation ring over $k$, with uniformiser $\pi$ and field of fractions $K$.
We will also write $X_R:= X \otimes_k R$, $X _K := X_R \otimes_R K$ and $X_m := X \otimes_k R/ \pi^m$ for any positive integer $m$.  We denote by $\iota : X_k \hookrightarrow X_R$ and $j : X_K \hookrightarrow X_R$ the closed and open immersions of the central and generic fibres of $X_R \to \Spec R$, respectively.

\begin{proposition}\label{prop6}\cite[Proposition 4.2]{Lo1}
Let $X$ be a smooth projective variety of dimension $n$ over $k$.  Given any object
\[
E_K \in \langle \Coh_{\leq d}(X_K), \Coh_{\geq n}(X_K)[1]\rangle
\]
where $0 \leq d < n$, there exists an object $\widetilde{E} \in D^b(X_R)$ such that:
\begin{itemize}
\item the generic fibre $j^\ast (\widetilde{E}) \cong E_K$ in $D^b(X_K)$;
\item the central fibre $L\iota^\ast (\widetilde{E}) \in \langle \Coh_{\leq d}(X_k), \Coh_{\geq n}(X_k)[1]\rangle$.
\end{itemize}
\end{proposition}


The other technical results in \cite{Lo1, Lo2} that generalise to our case of stabilities of type W1, which will be used to prove that they give open properties for complexes, are listed here:
\begin{itemize}
\item[(i)] All the results in \cite[Section 5]{Lo1} and \cite[Proposition 2.1]{Lo2} hold for $X$ of arbitrary dimension, and for hearts of the form $$\Ap_m :=\langle \Coh_{\leq d}(X_m), \Coh_{\geq d+1}(X_m)[1]\rangle;$$ these results have nothing to do with stability.  Also, \cite[Lemma 2.2]{Lo2} and \cite[Corollary 2.3]{Lo2} both hold for stabilities of type W1 on $X$ of any dimension.
\item[(ii)] \cite[Proposition 2.4]{Lo2} holds for $X$ of any dimension $n$,  when $\Coh_{3,1}$ is replaced with $\Coh_{n,1}$ in its statement.  The proof of the general case relies on Lemma \ref{lemma_a}.
\item[(iii)] \cite[Proposition 2.5]{Lo2} holds for $X$ of any dimension $n$,  when the category $$\langle \Coh_{\leq 0}(X_K), \Coh_{\geq 3}(X_K)[1]\rangle$$ is replaced with the category $$\langle \Coh_{\leq 0}(X_K), \Coh_{\geq n}(X_K)[1]\rangle,$$ and  $\Coh_{3,1}$ is replaced with $\Coh_{n,1}$ in its statement.
\item[(iv)] \cite[Proposition 2.6]{Lo2} holds for $X$ of any dimension $n$ and for stabilities of type W1,  when $\Coh_{3,1}$ is replaced with $\Coh_{n,1}$ in its statement; in the proof, the use of the reduced Hilbert polynomial $p_{3,1}$ is replaced with $p_{n,1}$.
\item[(v)] All the results in \cite[Section 2.2]{Lo2} hold for any heart of the form
\[
  \langle \Coh_{\leq d}(X), \Coh_{\geq d+1}(X)[1]\rangle \subset D(X),
\]
where $X$ is of arbitrary dimension $n$ and $1 \leq d \leq n$.  (These results only depend on those in \cite[Section 5]{Lo1}; see (i) above.)
\end{itemize}

As a consequence of (v) above, we have the following valuative criterion for universal closedness for stabilities of type W1, which generalises \cite[Theorem 2.23]{Lo2} to higher dimensions:

\begin{theorem}[Valuative criterion for universal closedness]\label{theorem2}
Fix any polynomial stability $\sigma$ of type W1.  Then, given any $\sigma$-semistable object $E_K \in \Ap (X_K)$ such that $ch_0 (E_K)\neq 0$, there exists $E \in D^b(X_R)$, a flat family of objects in $\Ap$ over $\Spec R$, such that $j^\ast E \cong E_K$ and $L\iota^\ast E$ is $\sigma$-semistable.
\end{theorem}

With the same proof as in \cite{Lo2}, we  immediately obtain the following result, generalising \cite[Proposition 3.3]{Lo2}:

\begin{corollary}[Openness of stabilities of type W1]\label{coro2}
Let $S$ be a Noetherian scheme over $k$, and $E \in D^b (X \times_{\Spec k} S)$ be a flat family of objects in $\Ap$ over $S$ with $ch_0 \neq 0$.  Let $\sigma$ be a polynomial stability of type W1 on $D^b(X)$, and suppose $s_0 \in S$ is a point such that $E_{s_0}$ is $\sigma$-semistable.  Then there is an open set $U \subseteq S$ containing $s_0$ such that for all points $s \in U$, the fibre $E_s$ is $\sigma$-semistable.
\end{corollary}

%

\subsection{Openness of $H^{-1}$ being reflexive}\label{subsection-reflexive}

Given a polynomial stability $\sigma$ of type W1, since $\sigma$-semistability is an open property for flat family of complexes, we have a moduli stack $\mathcal M^\sigma$ of $\sigma$-semistable complexes.  In order to send the objects parametrised by $\mathcal M^\sigma$ to semistable sheaves via Fourier-Mukai transform using the results from Section \ref{complexes}, we need to restrict to an open substack of $\sigma$-semistable objects $E$ where $H^{-1}(E)$ is a reflexive sheaf.  To this end, we will show the following:

\begin{theorem}\label{theorem3}
For flat families of 2-term complexes $E \in D(X)$ satisfying:
 \begin{itemize}
 \item $H^{-1}(E)$ is torsion-free,
 \item $H^0(E) \in \Coh_{\leq 0}(X)$,
 \item $H^i(E)=0$ for all $i \neq -1, 0$, and
 \item $\Hom (\Coh_{\leq 0}(X),E)=0$,
 \end{itemize}
the property that $H^{-1}(E)$ is a reflexive sheaf is an open property.
\end{theorem}

The proof of Theorem \ref{theorem3} will consist of two steps:
\begin{itemize}
\item[Step 1.]  We show, that for complexes $E$ satisfying the hypotheses of Theorem \ref{theorem3}, the property of $H^{-1}(E)$ being reflexive is equivalent to the following dimension requirements on the cohomology sheaves of the derived dual $E^\vee$:
    \begin{align}
      \dimension H^{n-1}(E^\vee) &\leq 0, \notag \\
      \dimension H^{n-2}(E^\vee) & \leq 1, \notag\\
      \vdots \notag \\
      \dimension H^2 (E^\vee ) & \leq n-3. \label{eq26}
    \end{align}
\item[Step 2.] We show that the requirements \eqref{eq26} form an open condition for flat families of complexes satisfying the hypotheses of Theorem \ref{theorem3}.
\end{itemize}

We begin with the easy observation:
\begin{lemma}\label{lemma27}
For a 2-term complex $E$ such that $H^{-1}(E)$ has homological dimension at most $n-1$ and $H^i (E)=0$ for all $i \neq -1, 0$, we have the equivalence
\[
\Hom_{D(X)} (\Coh_{\leq 0}(X), E)=0 \Leftrightarrow H^n(E^\vee)=0.
\]
\end{lemma}
\begin{proof}
Since $H^{-1}(E)$ has homological dimension at most $n-1$, we have $E^\vee \in D^{[0,n]}_{\Coh (X)} (X)$.  Note that $\Hom (\Coh_{\leq 0}(X),E)=0$ is equivalent to $\Hom (k_x, E)=0$ for all $x \in X$, where $k_x$ denotes the skyscraper sheaf of length one supported at the closed point $x$.  Now, for any $x \in X$ we have $\Hom (k_x , E) \cong \Hom (E^\vee, k_x[-n])$, and so the lemma follows.
\end{proof}

\begin{corollary}\label{coro3}
If $E \in \Ap$ is a $\sigma$-semistable object where $\sigma$ is of type W1, then $H^n(E^\vee)=0$, i.e.\ the right-most cohomology of $E^\vee$ is at degree $n-1$ or lower.
\end{corollary}
\begin{proof}
This follows from Lemma \ref{lemma27}, and the fact that $\phi (\rho_0) > \phi (-\rho_n)$ for stabilities of type W1.
\end{proof}

Now, we use the characterisation of reflexive sheaves in \cite[Section 1.1]{HL} to finish Step 1:

\begin{lemma}\label{lemma29}
For a 2-term complex $E \in D(X)$ such that $H^{-1}(E)$ is torsion-free, $H^i(E)=0$ for $i\neq -1, 0$ and such that $H^n(E^\vee)=0$, we have that $H^{-1}(E)$ is reflexive if and only if the conditions \eqref{eq26} are satisfied.
\end{lemma}

\begin{proof}
Since $E$ has cohomology only at degrees $-1$ and $0$, it fits in an exact triangle in $D(X)$
\[
 H^{-1}(E)[1] \to E \to H^0(E) \to H^{-1}(E)[2].
\]
Dualising, we obtain the exact triangle
\begin{equation}\label{eq29}
  (H^0(E))^\vee \to E^\vee \to (H^{-1}(E)[1])^\vee \to (H^0(E))^\vee [1],
\end{equation}
the long exact sequence of cohomology of which gives us the isomorphisms
\begin{equation}\label{eq27}
H^i (E^\vee) \cong H^i ((H^{-1}(E)[1])^\vee) \cong \EExt^{i-1} (H^{-1}(E),\OO_X) \text{ for $1 \leq i \leq n-2$},
\end{equation}
 as well as the exact sequence
\begin{equation}\label{eq28}
0 \to H^{n-1}(E^\vee) \to H^{n-1} ( (H^{-1}(E)[1])^\vee ) \to H^n (H^0(E)^\vee) \to 0.
\end{equation}
Note that the middle term $H^{n-1} ( (H^{-1}(E)[1])^\vee )$ in \eqref{eq28} is isomorphic to the sheaf $\EExt^{n-2}(H^{-1}(E),\OO_X)$.

Now, from \cite[Proposition 1.1.10]{HL}, we know that $H^{-1}(E)$, being a torsion-free sheaf, is reflexive if and only if $\dimension \EExt^q (H^{-1}(E),\OO_X) \leq n-q-2$ for all $q>0$, i.e.\ if and only if $\dimension \EExt^q (H^{-1}(E),\OO_X) \leq n-q-2$ for $1 \leq q \leq n-2$; that $\EExt^{n-1} (H^{-1}(E),\OO_X)=0$ follows from $H^n(E^\vee)=0$ and the long exact sequence of cohomology of \eqref{eq29}, while $\EExt^n (H^{-1}(E),\OO_X)=0$ follows from $H^{-1}(E)$ being torsion-free.  From the isomorphisms \eqref{eq27}, we have that $\dimension \EExt^q (H^{-1}(E),\OO_X) = \dimension H^{q+1} (E^\vee)$ for $1 \leq q \leq n-3$; that $\dimension \EExt^{n-2} (H^{-1}(E),\OO_X) = \dimension H^{n-1}(E^\vee)$ follows from the exact sequence \eqref{eq28} and the observation that $H^n(H^0(E)^\vee)$ is a 0-dimensional sheaf.  The lemma then follows.
\end{proof}

Consider the following conditions for complexes $E \in D^{\leq 0}_{\Coh (X)}(X)$:
\begin{align}
  \dimension H^0(E) &\leq 0, \notag\\
  \dimension H^{-1}(E) &\leq 1, \notag\\
  \vdots \notag \\
  \dimension H^{-n+3} (E) & \leq n-3. \label{eq30}
\end{align}
These are the same conditions as \eqref{eq26}, except that $E^\vee$ has been replaced by $E$, and the indices have been shifted.  The following lemma completes Step 2 above:

\begin{lemma}\label{lemma28}
The conditions \eqref{eq30} form an open property for flat families of complexes in $D^{\leq 0}_{\Coh (X)}(X)$.
\end{lemma}

\begin{proof}
Let $S$ be a noetherian scheme, and suppose $E \in D(X \times S)$ is an $S$-flat family of complexes in $D^{\leq 0}_{\Coh (X)}(X)$.  By using a flattening stratification on the cohomology sheaves of $E$ and semicontinuity, we see that the locus of $S$ over which the fibres of $E$ satisfy \eqref{eq30} is a constructible set.  It remains to show that this locus is stable under generisation.  To this end, let us suppose that $S = \Spec R$ where $R$ is a discrete valuation ring, and that $L\iota^\ast E$ satisfies the conditions \eqref{eq30}.  We need to show that $j^\ast E$ also satisfies \eqref{eq30}.
Suppose $E$ is represented by the complex
\[
  E^\bullet = [ \cdots \to  E^i \overset{d^i}{\to} \cdots \to E^{-1} \overset{d^{-1}}{\to} E^0 \to 0 \to \cdots  ]
\]
where $E^i=0$ for $i>0$.  Consider the spectral sequence
\begin{equation}\label{eq34}
E_2^{p,q} := L^p \iota^\ast (H^q (E)) \Rightarrow L^{p+q}\iota^\ast (E).
\end{equation}
Since $L^0\iota^\ast H^0(E) \in \Coh_{\leq 0}(X_k)$ by assumption, by semicontinuity we have that $j^\ast H^0(E) \cong H^0 (j^\ast E) \in \Coh_{\leq 0}(X_K)$.  Also, since
\begin{equation}\label{eq35}
  \text{supp}(L\iota^\ast F) = \text{supp}(\iota^\ast F) \text{ for any $F \in \Coh (X_R)$},
\end{equation}
it follows that $L^i \iota^\ast H^0(E) \in \Coh_{\leq 0}(X_k)$ for all $i$.

We now proceed by induction to show that $L^p \iota^\ast H^q (E) \in \Coh_{\leq -q}(X_k)$ for all $p\leq 0$ and $-n+3 \leq q \leq 0$.  The case $q=0$ is already checked above.  Suppose $d \leq 0$ is an integer such that, for all $d \leq m \leq 0$, we have $L^p \iota^\ast H^m (E) \in \Coh_{\leq -m}(X_k)$ for all $p$.  We want to show that
\begin{equation}\label{eq36}
  L^p \iota^\ast H^{d-1}(E) \in \Coh_{\leq -d+1}(X_k) \text{\quad for all $p \leq 0$}.
\end{equation}

Now, we have
\begin{align*}
  \dimension L^0 \iota^\ast H^{d-1}(E) &= \max{ \{ \dimension (\image (d_2^{-2,d})), \dimension E_3^{0,d-1} \}}, \\
  \dimension E_3^{0,d-1} &= \max{ \{ \dimension (\image (d_3^{-3,d+1})),\dimension E_4^{0,d-1}\}}, \\
  &\vdots
\end{align*}
 On the other hand, we have:
\begin{itemize}
\item $\dimension (\image (d_s^{-s,d-2+s})) \leq -(d-2+s)$ for all $s \geq 2$ from our induction hypothesis,
\item $E_t^{0,d-1}=E_\infty^{0,d-1}$  for $t \geq -(d-1)+2$, and
\item $E_\infty^{0,d-1} \in \Coh_{\leq -d+1}(X_k)$ by assumption.
\end{itemize}
Putting all these together, we get that $\dimension L^0 \iota^\ast H^{d-1}(E) \leq -d+1$.
Applying \eqref{eq35} once more, we obtain \eqref{eq36}.

In particular, we have shown that $L^0\iota^\ast H^q (E) \in \Coh_{\leq -q}(X_k)$ for all $-n+3 \leq q \leq 0$.  By semicontinuity, we have $H^q (j^\ast E) \cong j^\ast H^q (E) \in \Coh_{\leq -q}(X_K)$ for all $-n+3 \leq q \leq 0$, thus proving the lemma.
\end{proof}

\begin{proof}[Proof of Theorem \ref{theorem3}]
The theorem now follows from Lemmas \ref{lemma27}, \ref{lemma29} and \ref{lemma28}.
\end{proof}

Lemma \ref{lemma32} below  is likely well-known, but we note that the proof of Lemma \ref{lemma28} can be easily adapted to show it.  Following the notation in \cite[Section 1.1]{HL}, given a coherent sheaf $E$ of dimension $d$ on a smooth projective variety $X$ of dimension $n$, if we write $c := n-d$ as the codimension of $E$, then we say $E$ satisfies condition $S_{k,c}$ (where $k \geq 0$) if:
\begin{equation*}
  \text{depth}(E_x) \geq \min{\{ k, \dimension (\mathscr{O}_{X,x}) - c\}} \text{ for all } x \in \text{supp}(E).
\end{equation*}
This generalises Serre's condition $S_k$.

\begin{lemma}\label{lemma32}
For a flat family of coherent sheaves on $X$, being $S_{k,c}$ is an open property.
\end{lemma}

\begin{remark}\label{remark12}
For a torsion-free sheaf on $X$, being $S_2$ is equivalent to being $S_{2,0}$, which in turn is equivalent to being reflexive by \cite[Proposition 1.1.10]{HL}.  Therefore, by Lemma \ref{lemma32}, being reflexive is an open property for a flat family of coherent sheaves on $X$.
\end{remark}


\begin{lemma}\label{lemma30}
Suppose $E \in D(X)$ is such that $E^\vee \in D^{[0,n]}_{\Coh (X)}(X)$.  Then we have the vanishing
\begin{equation}\label{eq37}
\Hom (\Coh_{\leq 1}(X),E)=0
\end{equation}
if and only if $H^n (E^\vee)=0$ and $H^{n-1}(E^\vee)\in \Coh_{\leq 0}(X)$.  Therefore, the vanishing \eqref{eq37} is an open property for flat families of complexes $E$ satisfying $E^\vee \in D^{[0,n]}_{\Coh (X)}(X)$.
\end{lemma}

In Theorem \ref{theorem5} below, we will show how Lemma \ref{lemma30} implies the existence of moduli stacks for objects in the category $\mathfrak{D}$ described in \cite[Section 7.2]{BMTI}.

\begin{proof}
Take any $E \in D(X)$ such that  $E^\vee \in D^{[0,n]}_{\Coh (X)}(X)$.  Suppose we have $\Hom (\Coh_{\leq 1}(X),E)=0$.  In particular, we have $\Hom (\Coh_{\leq 0}(X),E)=0$; this, together with the fact that the right-most cohomology of $E^\vee$ is at degree $n$, implies $H^n (E^\vee)=0$.  Next, suppose $\dimension H^{n-1} (E^\vee) \geq 1$.  Then there exists a nonzero morphism of sheaves $\al : H^{n-1}(E^\vee ) \to T$ where $T$ is a pure 1-dimensional sheaf.  Let $\theta_T : T \to T^{DD}$ be the natural map of sheaves as in \cite[Lemma 1.1.8]{HL} (here, $(-)^D$ is the dual in the sense of \cite[Definition 1.1.7]{HL}); since $T$ is pure, we have that $\theta_T$ is an injection by \cite[Proposition 1.1.10]{HL}. In particular $\theta_T$ is an isomorphism, since if $T$ is a pure $d$-dimensional sheaf, we have a short exact sequence $0 \to T \to T^{DD} \to Q \to 0$, where $Q$ is at most $(d-2)$-dimensional. Then the composition $\theta_T \al$ is nonzero.

By \cite[Proposition 1.1.10]{HL} again, we see that $T^D$ itself is 1-dimensional, reflexive, $S_{2,n-1}$, and pure.  Then by \cite[Proposition 1.1.6(ii)]{HL}, we have the vanishing $\EExt^n (T^D,\OO_X)=0$, and so
\[
(T^D)^\vee \otimes \omega_X \cong \EExt^{n-1} (T^D,\omega_X) [-n+1] \cong T^{DD}[-n+1].
\]
Therefore, we have
\begin{align*}
0 \neq \theta_T \al &\in \Hom (E^\vee, T^{DD}[-n+1]) \\
  &\cong \Hom (E^\vee, (T^D)^\vee \otimes \omega_X) \\
  &\cong \Hom (T^D \otimes \omega_X^\ast, E)^\ast,
\end{align*}
contradicting our assumption that $\Hom (\Coh_{\leq 1}(X),E)=0$.  Hence $H^{n-1}(E^\vee)$ must lie in $\Coh_{\leq 0}(X)$.

For the converse, suppose $E \in D(X)$ satisfies $E^\vee \in D^{[0,n]}_{\Coh (X)}(X)$, and  is such that $H^n(E^\vee )=0$ and $H^{n-1}(E^\vee) \in \Coh_{\leq 0}(X)$.  We want to show the vanishing $\Hom (\Coh_{\leq 1}(X),E)=0$.  By the same argument as in the proof of Lemma \ref{lemma27}, we know that $H^n (E^\vee)=0$ implies $\Hom (\Coh_{\leq 0}(X),E)=0$.  So it remains to show that $\Hom (T,E)=0$ for any pure 1-dimensional sheaf $T$ on $X$.  Suppose there is a nonzero morphism $\al : T \to E$ for some pure 1-dimensional sheaf $T$.  Since $T$ is pure, it is $S_{1,n-1}$ \cite[Section 1.1]{HL}, and $\EExt^n (T,\OO_X)=0$ by \cite[Proposition 1.1.6(ii)]{HL}; as a result, we have $T^\vee \cong T^D[-n+1]$ where $T^D$ is again pure of dimension 1.  Hence
\[
0 \neq \al \in \Hom (T,E) \cong \Hom (E^\vee,T^\vee) \cong \Hom (E^\vee, T^D[-n+1])
\]
which is impossible since $H^n (E^\vee)=0$ and $H^{n-1}(E^\vee) \in \Coh_{\leq 0}(X)$.

The last part of the lemma follows from semicontinuity for sheaves.
\end{proof}

\begin{lemma}
Let $F$ be a torsion-free sheaf on $X$.  Then
\[
  \text{$F$ is reflexive } \Leftrightarrow \Hom (\Coh^{\leq n-2}(X),F[1])=0.
\]
\end{lemma}

\begin{proof}
Suppose $F$ is reflexive.  Then we can find a short exact sequence in $\Coh (X)$
\[
0 \to F \to E \to G \to 0
\]
where $E$ is locally free and $G$ is torsion-free \cite[Proposition 1.1]{SRS}.  Applying $\Hom_X (T,-)$ to the short exact sequence for an arbitrary $T \in \Coh^{\leq n-2}(X)$, we obtain the exact sequence
\[
\Hom (T,G) \to \Hom (T,F[1]) \to \Hom (T,E[1])
\]
where $\Hom (T,G)=0$ and $\Hom (T,E[1])\cong \Ext^1 (T,E) \cong \Ext^{n-1}(E,T\otimes \omega_X) \cong H^{n-1}(X,E^\ast \otimes T \otimes \omega_X)$, which vanishes because $T$ is supported in dimension at most $n-2$.  Hence $\Hom (\Coh^{\leq n-2}(X),F[1])=0$.

For the converse, suppose $F$ is a torsion-free sheaf satisfying the vanishing $\Hom (\Coh^{\leq n-2}(X),F[1])=0$.  Then we have a short exact sequence in $\Coh (X)$
\[
0 \to F \to F^{\ast \ast} \to T \to 0
\]
where $T \in \Coh^{\leq n-2}(X)$.  This short exact sequence must be split by our hypothesis, forcing $T=0$, i.e.\ $F$ is reflexive.
\end{proof}

The following technical result will also be needed when it comes to constructing moduli stacks in the next section:

\begin{lemma}\label{lemma31}
Let $\sigma$ be a polynomial stability of type W1, and  $E \in \Ap$  a $\sigma$-semistable object.  If $H^{-1}(E)$ is reflexive, then $\Hom (\Coh_{\leq 1}(X),E)=0$.  When $X$ is  a threefold, the converse also holds.
\end{lemma}
\begin{proof}
By Corollary \ref{coro3} and Lemma \ref{lemma30}, we just have to show $H^{n-1}(E^\vee) \in \Coh_{\leq 0}(X)$, which indeed holds by Lemma \ref{lemma29}.  That the converse holds when $X$ is a threefold is easy.
\end{proof}

Combining Lemma \ref{lemma31} and a couple of results from \cite{Lo4}, we obtain:

\begin{lemma}\label{lemma34}
Let $X$ be a threefold, and let $\sigma, \tilde{\sigma}$ be polynomial stabilities of type W1 and W2 on $D(X)$, respectively.  Let $E \in \Ap$ be a $\sigma$-semistable object where $ch_0(E)\neq 0$, and $ch_0(E), ch_1(E)$ are coprime.  Then the following are equivalent:
\begin{itemize}
\item[(i)] $E$ is $\tilde{\sigma}$-stable;
 \item[(ii)] $E$ is $\tilde{\sigma}$-semistable;
 \item[(iii)] $\Hom (\Coh_{\leq 1}(X),E)=0$.
\end{itemize}
\end{lemma}

\begin{proof}
Suppose $E$ is as described.  Suppose $E$ is also $\tilde{\sigma}$-semistable.  Then $H^{-1}(E)$ is reflexive by \cite[Lemma 3.2]{Lo3}, and so we have $\Hom (\Coh_{\leq 1}(X),E)=0$ by Lemma \ref{lemma31}.  Hence (ii) implies (iii).

Now, suppose  the vanishing $\Hom (\Coh_{\leq 1}(X),E)=0$ holds.  By \cite[Lemma 3.10]{Lo3}, we have that $H^{-1}(E)$ is $\mu$-stable.  By \cite[Lemma 3.5]{Lo3}, we get that $E$ is $\tilde{\sigma}$-stable, hence $\tilde{\sigma}$-semistable.  Hence (iii) implies (i).

Hence (i), (ii) and (iii) are equivalent.
\end{proof}

In \cite[Section 7.2]{BMTI}, Bayer-Macr\`{i}-Toda considers a category $\mathfrak{D}$ of two-term complexes, that appear to be closely related to tilt-semistable objects (see \cite[Lemmas 7.2.1, 7.2.2]{BMTI}).  For the following proposition, let $R$ be a discrete valuation ring over $k$, with $\iota, j$ as before.  Also, let $\Bob$ be as defined in \cite[Sections 3.1]{BMTI} and let $\mathfrak{D} \subset \Bob$ be the set of objects $E$ satisfying one of the following conditions \cite[Sections 7.2]{BMTI}:
\begin{itemize}
\item[(a)] $H^{-1}(E)=0$ and $H^0(E)$ is a pure sheaf of dimension $\geq 2$ that is slope-semistable with respect to $\omega$.
\item[(b)] $H^{-1}(E)=0$ and $H^0(E)$ is a sheaf of dimension $\leq 1$.
\item[(c)] $H^{-1}(E)$ is a torsion-free slope-semistable sheaf and $H^0(E) \in \Coh^{\leq1}(X)$. If $\mu_{\omega,B}(H^{-1}(E))<0$, we have $\Hom(\Coh^{\leq1}(X),E)=0$.
\end{itemize}

\begin{theorem}\label{theorem5}
Suppose $X$ is a threefold, and $ch$ a fixed Chern character where $ch_0 \neq 0$. For flat families of objects $E$ in $\Ap$ of Chern character $ch$, the property that $H^{-1}(E)$ is $\mu$-semistable is an open property.  As a consequence, objects of Chern character $ch$ in $\mathfrak{D}$ form a moduli stack.
\end{theorem}

\begin{proof}
The argument for openness is the same as the second half of the proof of \cite[Proposition 3.1]{Lo2}, except that here we use the slope $\mu$ instead of the reduced Hilbert polynomial $p_{3,1}$.

For the second assertion of the lemma, note that  being in the heart $\Bob$ is an open property for complexes (by \cite[Example 1(2), Appendix A]{ABL}), as is being in $\Ap$ (by \cite[Lemma 3.14]{TodaLSO}).  Therefore, for a fixed Chern character $ch$ where $ch_0 \neq 0$, we have a moduli stack of objects $E \in \Bob$ with Chern character $ch$ such that $H^0(E) \in \Coh_{\leq 1}(X)$.  By Lemma \ref{lemma30}, the property that $\Hom (\Coh_{\leq 1}(X), E)=0$ is also an open property for flat families of complexes in $\Ap$.  Therefore, the moduli stack of objects of type (c) in $\mathfrak D$ exists, regardless of whether $\mu (H^{-1}(E))<0$ or not.
\end{proof}


\begin{remark}\label{remark15}
Suppose $X$ is a threefold, $ch_0 \neq 0, ch_1$ are coprime with $ch_1\cdot \omega^2/ch_0 <0$, and $\sigma, \tilde{\sigma}$ are polynomial stabilities of types W1 and W2, respectively.  Suppose $E \in D(X)$ is such that $H^0(E) \in \Coh_{\leq 0}(X)$.  Then by Lemmas \ref{lemma_a} and \ref{lemma34}, we have the following implications:
\[
\text{$E \in \mathfrak D$ is of type (c) } \Leftrightarrow \text{ $E$ is $\tilde{\sigma}$-stable } \Rightarrow \text{ $E$ is $\sigma$-stable}
\]
In other words, the moduli of objects in $\mathfrak {D}$ of type (c) with $\mu_{\omega, B}<0$ and 0-dimensional $H^0$ can be described as the moduli of $\tilde{\sigma}$-stable objects.
\end{remark}

\subsection{An open immersion of moduli stacks}\label{subsection-open-immersion}

Fix a Chern character $ch$ where $ch_0 \neq 0$, and a polynomial stability $\sigma$ of type W1.  By Corollary \ref{coro2}, there exists a moduli stack $\MM^\sigma$ of $\sigma$-semistable objects of Chern character $ch$.   By Lemma \ref{lemma_a} and Theorem \ref{theorem3}, we have an open substack $\MM^{\sigma, R} \subseteq \MM^\sigma$ parametrising the complexes $E \in D(X)$ such that $H^{-1}(E)$ is reflexive.

For any Noetherian scheme $B$ over the ground field $k$ and any $B$-flat family of complexes $E_B$ on $X$, define the following property (P) for fibres $E_b$ of $E_B$, $b \in B$:
    \begin{itemize}
      \item[(P)] The restriction $(H^{-1}(E_b))|_s$ of the cohomology sheaf $H^{-1}(E_b)$ to the fibre $\pi^{-1}(s)$ is a stable sheaf for a generic point $s \in S$.
    \end{itemize}
We have:
\begin{proposition}\label{prop1}
Property (P) is an open property for flat families of complexes in $\Ap$.
\end{proposition}

\begin{proof}
The proof of the threefold case, which was done in two parts in \cite[Lemma 3.5, Lemma 3.6]{Lo5}, generalises to the case of $n \geq 3$ in a straightforward manner.
\end{proof}
 As a result of Proposition \ref{prop1},  we have another open substack $\MM^{\sigma, R, P} \subset  \MM^{\sigma, R}$ consisting of complexes $E \in D(X)$ in $\MM^{\sigma, R}$ such that $H^{-1}(E)|_s$ is a stable sheaf for a generic point $s \in S$.

Overall, we have the following open immersions of stacks of complexes:
\begin{equation}\label{eq38}
   \MM^{\sigma, R, P} \subset \MM^{\sigma, R} \subset \MM^\sigma.
\end{equation}
Suppose we have fixed our Chern character $ch$ so that, for the complexes $E \in D(X)$ parametrised by $\MM^\sigma$, we have $\mu (H^{-1}(E)) < b/a$.  Then:

\begin{theorem}\label{thm_a}
Let $\pi$ be as in Theorem \ref{main1}. We have an open immersion of moduli stacks
\begin{equation}\label{eq39}
\xymatrix{
\MM^{\sigma, R, P}  \arinj[r]^(.6){} & \MM^s \ ,
}
\end{equation} where $\MM^s$ denotes the moduli stack of Gieseker stable torsion-free sheaves on Y , with respect to some polarisation.
\end{theorem}

The proof of this theorem is the same as that of the threefold case, namely \cite[Theorem 3.1]{Lo5}, except for the very last step where we invoke Lemma \ref{lemma31} to show the vanishing $\Hom (\Coh_{\leq 1}(X), E)=0$.  We reproduce the proof here since it is short, and also for clarity.

\begin{proof}
Take any $E \in \mathcal{A}^p$ corresponding to a point in $\MM^{\sigma, R, P}$.   By Corollary \ref{coro1}, we know that $H^{-1}(E)$ is $\Psi$-WIT$_1$ and $\widehat{H^{-1}(E)}$ is torsion free.  The Fourier-Mukai transform $\Psi$ takes the exact triangle in $D(X)$
\[
H^{-1}(E)[1] \to E \to H^{0}(E) \to H^{-1}(E)[2]
\]
to a short exact sequence of coherent sheaves on $Y$
\[
0 \to \widehat{H^{-1}(E)} \to \widehat{E} \to \widehat{H^{0}(E)} \to 0 \ .
\]
Since  $H^0(E)$ is supported on a finite number of points, it follows that  $\widehat{H^{0}(E)}$ is supported on  a finite number of fibres by \cite[Proposition 6.1]{FMNT}. Using \cite[Lemma 9.5]{BMef} we know that $\widehat{H^{-1}(E)}$ restricts to a stable sheaf on a generic fibre. Therefore $\widehat{E}$ is stable when restricted to a generic fibre of $\widehat{\pi}$.

 By \cite[Lemma 2.1]{BMef}, if $\widehat{E}$ is a torsion-free sheaf that restricts to a stable sheaf on a generic fibre, then  $\widehat{E}$ is stable with respect to a suitable polarisation on $Y$. Therefore it remains to show that $\widehat{E}$ is torsion free.

Suppose $\widehat{E}$ is not torsion-free. Denote by $T$ its maximal torsion subsheaf, which would be nonzero. Since $\widehat{H^{-1}(E)}$ is torsion free, we have an injection $T \hookrightarrow \widehat{H^{0}(E)}$ and $T$ is $\Phi$-WIT$_1$. The inclusion $T \hookrightarrow \widehat{E}$ gives a nonzero element in
\[
\Hom_Y(T, \widehat{E}) \cong \Hom_X(\Phi T, \Phi \widehat{E}) \cong \Hom_X(\widehat{T}, E) \ .
\]
Since $\widehat{H^0(E)}$ is supported on a finite number of fibres, so is $T$, and so $\dimension T \leq 1$.  However, by Lemma \ref{lemma31}, we have $\Hom (\Coh_{\leq 1}(X), E)=0$, a contradiction.  Hence $\widehat{E}$ must have been torsion-free to begin with.
\end{proof}

\subsection{Comparison with the threefold case}\label{subsection-comparison}

Let $\sigma$ be a polynomial stability of type W1 throughout the rest of this section.

In the theorem we have just proved, Theorem \ref{thm_a}, we embed the moduli stack $\MM^{\sigma, R, P}$ into a moduli stack of stable sheaves.  Recall that $\MM^{\sigma, R, P}$ parametrises $\sigma$-semistable objects $E\in\Ap$ (of fixed $ch$ where $ch_0 \neq 0$) such that   $H^{-1}(E)$ is torsion-free and reflexive, and $E$ satisfies property (P).

On the other hand, in \cite[Theorem 3.1]{Lo5}, where $X$ is a threefold, we embed a moduli stack $\MM^{\sigma, \tilde{\sigma}, P}$ into a moduli stack of stable sheaves.  There, $\MM^{\sigma, \tilde{\sigma}, P}$ parametrises $\sigma$-semistable objects $E \in \Ap$ such that $E$ is  also $\tilde{\sigma}$-semistable ($\tilde{\sigma}$ being a polynomial stability of type W2), and satisfies property (P).

Given an object $E \in \Ap$ on a threefold $X$, if $E$ is $\tilde{\sigma}$-semistable, then $H^{-1}(E)$ is torsion-free and reflexive by \cite[Lemma 3.2]{Lo3}.  Therefore, when $X$ is a threefold, the stack $\MM^{\sigma, R, P}$ contains $\MM^{\sigma,\tilde{\sigma}, P}$ as a substack, i.e.\ our Theorem \ref{thm_a} appears more general than \cite[Theorem 3.1]{Lo5}.  By Lemma \ref{lemma34}, however, we know that  $\MM^{\sigma, R, P}$ and $\MM^{\sigma,\tilde{\sigma}, P}$ coincide when we assume that $ch_0 \neq 0, ch_1$ are coprime.



\section{Moduli of rank-one torsion-free sheaves}
\label{moduli2}

Throughout this section, suppose that $n =\dimension X \geq 3$, and  $\pi$ is as in Theorem \ref{main1}.  Let us first recall the following theorem, which holds in higher dimensions with the same proof:

 \begin{theorem}\cite[Theorem 4.1]{Lo5}\label{thm1}
The functor $\Psi$ induces an equivalence  between the following two categories:
 \begin{itemize}
 \item[(i)] the category $\mathcal C_X$ of objects $E$ in
 \[
 \langle \mathcal B_X\cap W_{0,X}, \mathcal B_X^\circ \cap W_{1,X} [1] \rangle
  \]
  satisfying
  \[
  \Hom (\mathcal B_X \cap W_{0,X},E)=0,
   \]
   such that $H^{-1}(E)$ has nonzero rank, $\mu (H^{-1}(E)) < b/a$, and $H^{-1}(E)$ restricts to a stable sheaf on the generic fibre of $\pi$;
 \item[(ii)] the category $\mathcal C_Y'$ of torsion-free sheaves $F$ on $Y$ such that $\mu (F) > -c/a$, and $F$ restricts to a stable sheaf on the generic fibre of $\widehat{\pi}$, and such that in the unique short exact sequence
 \[
   0 \to A \to F \to B \to 0
 \]
 where $A$ is $\Phi$-WIT$_0$ and $B$ is $\Phi$-WIT$_1$, we have $B \in \mathcal B_Y$.  (Note that, this is equivalent to requiring $B$ to be  a torsion sheaf by Lemma \ref{lemma3}.)
 \end{itemize}
 Under the above equivalence of categories, we have $A = \widehat{H^{-1}(E)}$ and $B = \widehat{H^0(E)}$.
 \end{theorem}

Let $\overline{\mathcal C_Y'}$ denote the category of torsion-free sheaves $F$ on $Y$ such that $\mu (F) > -c/a$, and $F$ restricts to a stable sheaf on the generic fibre of $\widehat{\pi}$.  Note that $\mathcal C_Y' \subseteq \overline{\mathcal C_Y'}$.

Take any $F$ in $\overline{\mathcal C_Y'}$.  Consider the short exact sequence
\[
0 \to A \to F \to B \to 0
\]
where $A$ is $\Phi$-WIT$_0$ and $B$ is $\Phi$-WIT$_1$.  Suppose $F$ is rank-one; then either $r(A)=0$ or $r(A)=1$.  If $r(A)=0$, then $A=0$ since $F$ is torsion-free.  Then $F$ is $\Phi$-WIT$_1$, and so $\mu (F) \leq -c/a$ by Lemma \ref{lemma2} (or, rather, its analogue on $Y$), a contradiction.  Hence $r(A)=1$, in which case $B$ must be torsion, i.e.\ $F$ lies in $\mathcal C_Y'$.  Thus we have:

\begin{lemma}\label{lemma22}
The rank-one objects of $\mathcal C_Y'$ are the same as the rank-one objects in $\overline{\mathcal C_Y'}$, which are exactly the rank-one torsion-free sheaves with $\mu > -c/a$ on $Y$.
\end{lemma}

\begin{remark}\label{remark5}
Suppose an object $E\in \mathcal C_X$ maps to an object $F \in \mathcal C_Y'$ under $\Psi$.  Then $r(F)=r(\widehat{H^{-1}(E)})=-r(\Psi (H^{-1}(E)) = r(\Psi (E))=-b\cdot r(E) + a\cdot d(E)$.
Hence, by Lemma \ref{lemma22}, the category of objects $E$ in $\mathcal C_X$ satisfying $-b \cdot r(E) + a\cdot d(E)=1$  form a moduli space that is isomorphic to the moduli of rank-one torsion-free sheaves on $Y$.
In other words, the three conditions for complexes $E$ on $X$ in the definition of $\mathcal C_X$ (except for those conditions on the Chern classes of $E$), namely
\begin{itemize}
\item $E \in \langle \mathcal B_X\cap W_{0,X}, \mathcal B_X^\circ \cap W_{1,X} [1] \rangle$;
\item $\Hom (\mathcal B_X \cap W_{0,X},E)=0$; and
\item $H^{-1}(E)$ restricts to a stable sheaf on the generic fibre of $\pi$
\end{itemize}
should, in some sense, correspond to a  of stability condition for complexes.
\end{remark}

\section{Pure Codimension-$1$ Sheaves}\label{pure}

In this section, we study coherent sheaves supported in codimension 1 via Fourier-Mukai transforms.  Our main goal is to produce Theorem \ref{theorem6}, an equivalence  between the category of line bundles of fibre degree 0 on $X$ and the category of line bundles supported on sections of the dual elliptic fibration $Y$, which generalises \cite[Corollary 5.9]{Lo5} to higher dimensions.  The proofs of many results in \cite[Section 5]{Lo5} leading to \cite[Corollary 5.9]{Lo5} only hold for elliptic surfaces or threefolds; we rewrite their proofs for higher dimensional elliptic fibrations.

Before we come to the results, let us write $\mathcal P$ for the universal family on $Y \times X$ as in \cite[Section 8.4]{BMef}, and write $\mathcal Q := R\HHom_{\OO_{Y \times X}} (\mathcal P, \pi_X^\ast \omega_X) [n-1]$, so that $\Psi$ can be taken as the integral transform $D(X) \to D(Y)$ with kernel $\mathcal Q$.  Note that \cite[Lemma 8.4]{BMef} holds whenever $\dimension S \geq 1$  (as is \cite[Lemma 6.5]{FMTes}).  Therefore, $\mathcal Q$ is a sheaf that is flat over both $X$ and $Y$.

Throughout this section, let $S$ be of any dimension at least 1, i.e.\ the dimension of $X$ is at least 2.

\begin{lemma}\label{lemma14}
If $F$ is a pure codimension-1 sheaf on $Y$ that is flat over $S$, then $F$ is $\Phi$-WIT$_0$, lies in $\mathcal B_Y^\circ$, and $\widehat{F}$ is torsion-free.
\end{lemma}

\begin{proof}
Suppose $F$ is as described.  We first show that $F$ is $\Phi$-WIT$_0$.  The argument is essentially the same as that in \cite[Remark 5.6]{Lo5}, but we rewrite the proof slightly for clarity: by \cite[Lemma 6.5]{FMTes}, it suffices to show that $\Hom (F,\mathcal Q_x)=0$ for all $x \in X$.  Take any nonzero morphism of sheaves $F \overset{\al}{\to} \mathcal Q_x$ in $\Coh (Y)$, for any $x \in X$.  Then the support of  $\image (\al)$ is contained in $\widehat{\pi}^{-1}(\pi(x)) \cap \text{supp}(F)$.  Since $\mathcal Q_x$ is a stable 1-dimensional sheaf on $\widehat{\pi}^{-1}(\pi(x))$, the sheaf $\image (\al)$ cannot be 0-dimensional.  If $\image (\al)$ is a 1-dimensional sheaf, then $F|_{\pi(x)}$ is 1-dimensional, and by the flatness of $F$ over $S$, we see that $F$ has nonzero rank, a contradiction.  Hence $F$ is $\Phi$-WIT$_0$.

Next, we show that $F \in \mathcal B_Y^\circ$.  Take any nonzero $A \in \mathcal B_Y$, and consider a morphism $A \overset{\beta}{\to} F$ in $\Coh (Y)$.  Since $\image (\beta) \in \mathcal B_Y$, we can replace $A$ by $\image (\beta)$ and assume $\beta$ is an injection.  Since $F$ is a pure sheaf, $A$ itself must also be pure and of codimension 1.  This, along with $A \in \mathcal B_Y$, implies that the support of $A$ contains a fibre of $\widehat{\pi}$.  Therefore, the support of $F$ also contains a fibre of $\widehat{\pi}$; by flatness of $F$ over $S$, we get that $F$ has nonzero rank, again a contradiction.

Lastly, that $\widehat{F}$ is torsion-free follows from Lemma \ref{lemma7}.
\end{proof}

The following is an analogue of \cite[Proposition 5.7]{Lo5} when we do not require the dimension of $S$ to be 1-dimensional:

\begin{pro}\label{prop3}
Let $\mathcal D_X$ and $\mathcal D_Y'$ be the following categories:
\begin{align}\label{eq20}
\mathcal D_X := & \{ E \in \Coh (X) : \text{$E$ is torsion-free with positive rank, flat over $S$}, \\
& \mu (E)=b/a,\,   \Ext^1 (\mathcal B_X \cap W_{0,X},E)=0, \text{ and } E|_s \text{ is WIT$_1$} \text{ for all } s \in S \} \nonumber \\
\mathcal D_Y' := & \{ F \in \Coh (Y) : F \text{ is pure of codimension 1, flat over $S$}\} \ . \nonumber
\end{align} Then the functor $\Psi [1] : D(X) \to D(Y)$ induces an equivalence of categories $\mathcal D_X \rightarrow \mathcal D_Y'$.

\end{pro}

\begin{proof}
Take any nonzero $F \in \mathcal D_Y'$.  Since $F$ is flat over $S$, it cannot lie in $\mathcal B_Y$.  By Lemmas \ref{lemma7}, \ref{lemma12} and  \ref{lemma14}, we obtain that $F$ is $\Phi$-WIT$_0$, and that $\widehat{F}$ is torsion-free with positive rank, $\mu (\widehat{F})=b/a$ and $\Ext^1 (\mathcal B_X \cap W_{0,X},\widehat{F})=0$.  Since the  restriction of $F$ to each fibre $\widehat{\pi}^{-1}(s)$ is a 0-dimensional sheaf,  hence  $\Phi_s$-WIT$_0$, by \cite[Corollary 6.2]{FMNT}, we have that $\widehat{F}$ is flat over $S$.  Also by \cite[Corollary 6.2]{FMNT}, we have $\widehat{F}|_s$ is $\Psi_s$-WIT$_1$ for all $s \in S$.  Hence $\widehat{F} \in \mathcal D_X$.

Next, take any $E \in \mathcal D_X$.  By \cite[Corollary 6.2]{FMNT}, we know $E$ is $\Psi$-WIT$_1$ and $\widehat{E}$ is flat over $S$.  By Lemma \ref{lemma12}, the transform $\widehat{E}$ is a torsion sheaf but does not lie in $\mathcal B_Y$.  Hence $\widehat{E}$ is of codimension 1.  By Lemma \ref{lemma7}, $\widehat{E}$ does not have any subsheaf in $\mathcal B_Y$, and so does not have any subsheaf of codimension 2 or greater, i.e.\ $\widehat{E}$ is a pure sheaf.  Hence $\widehat{E}$ lies in $\mathcal D_Y'$.
\end{proof}

\begin{definition}\cite[Definitions 6.8, 6.10]{FMNT}
A Weierstrass fibration is an elliptic fibration $\pi:X \to S$ such that all the fibres of $\pi$ are geometrically integral Gorenstein curves of arithmetic genus 1, and there is a section $\sigma: S \to X$ of $\pi$ such that $\sigma(S)$ does not contain any singular point of the fibres.
\end{definition}
Note that, any fibre of a Weierstrass fibration as defined above necessarily has trivial dualising sheaf \cite[Section 1.1]{genus1}, and so a Weierstrass fibration in this sense satisfies the hypothesis of Theorem \ref{main1}.

\begin{remark}\label{remark2}
By \cite[Proposition 6.51]{FMNT}, when $b=0$ and $\pi, \widehat{\pi}$ are Weierstrass fibrations, every sheaf in $\mathcal D_X$ is fiberwise torsion-free and semistable.
\end{remark}

\begin{remark}\label{remark11}
Proposition \ref{prop3} reduces to the second equivalence of categories for elliptic surfaces in \cite[Proposition 5.7]{Lo5}.
\end{remark}


The following result was stated only for elliptic surfaces and threefolds in \cite{Lo5}, but its proof works as long as $\dimension S \geq 2$:

\begin{lemma}\cite[Lemma 5.8]{Lo5}\label{lemma15}
 Suppose $F$ is a pure codimension-1 sheaf on $Y$ that is flat over $S$.  Then $\widehat{\pi}$ restricts to a finite morphism $\widehat{\pi} : \text{supp}(F) \to S$, and $F \in \mathcal B_Y^\circ$. Furthermore, if $d(F)=1$, then $\text{supp}(F)$ is a section of $\widehat{\pi} : Y \to S$, and $F$ is a line bundle on $\text{supp}(F)$.
\end{lemma}


\begin{lemma}\label{lemma16}
Suppose $a=1,b=0$ and  $\pi$ is a Weierstrass fibration.  Take any $F \in \mathcal D_Y'$ such that $d(F)=1$.  Then $E := \widehat{F}$ fits in a short exact sequence in $\Coh (X)$
\begin{equation}\label{eq4}
  0 \to E \to E^{\ast \ast} \to T \to 0
\end{equation}
where $E^{\ast \ast}$ is a $\Psi$-WIT$_1$ line bundle whose semistability locus is all of $S$, and $T$ is also $\Psi$-WIT$_1$ and lies in $\mathcal B_X$.  Moreover, $E^{\ast \ast}$ lies in $\mathcal D_X$.
\end{lemma}

\begin{proof}
From the formula \eqref{eqn-rdunderPhi} and Proposition \ref{prop3}, we know that $E$ is a rank-one torsion-free sheaf.  Since the double dual $L := E^{\ast \ast}$ is a rank-one reflexive sheaf and $X$ is smooth, it is a line bundle \cite[Proposition 1.9]{SRS}.  Note that, since $T$ has codimension at least 2, we have $T \in \mathcal B_X$.  We now show that $L$ is $\Psi$-WIT$_1$:

Consider the short exact sequence in $\Coh (X)$:
\[
 0 \to L_0 \to L \to L_1 \to 0
\]
where $L_i \in W_{i,X}$.  Suppose $L$ is not $\Psi$-WIT$_1$; then $L_0 \neq 0$ and must be rank-one torsion-free, implying $L_1$ is $\Psi$-WIT$_1$ and torsion, and so $L_1 \in \mathcal B_X$ by Lemma \ref{lemma3}.  Now, applying $\Psi$ to \eqref{eq4} and taking the long exact sequence, we obtain an injection $0\to \Psi^0 (L) \to \Psi^0 (T)$.  Since $T \in \mathcal B_X$, by Lemma \ref{lemma5}, we have that $\Psi^0 (T)$ is a torsion sheaf.  Hence $\Psi^0 (L) = \widehat{L_0}$ and $\Psi^0(T)$ are both torsion, $\Phi$-WIT$_1$ sheaves, and must lie in $\mathcal B_Y$ by Lemma \ref{lemma3} again.  Hence $L_0 \in \mathcal B_X$, and $L$ itself lies in $\mathcal B_X$, which is a contradiction.  Therefore, we obtain that $L$ must be $\Psi$-WIT$_1$.  By Theorem \ref{main1}, we obtain $\Ext^1 (\mathcal B_X \cap W_{0,X}, L)=0$ as well.

On the other hand, applying $\Psi$ to the exact sequence \eqref{eq4} and then taking the long exact sequence, we obtain an injection $\Psi^0(T) \hookrightarrow\widehat{E} = F$.  Lemma \ref{lemma14}, however, tells us that $F \in \mathcal B^\circ_Y$.  Hence $\Psi^0(T)$ must vanish, i.e.\ $T$ is in fact $\Psi$-WIT$_1$.

It remains to show that the semistability locus of $L$ is all of $S$.  Take any closed point $s \in S$.  Then the restriction $L|_s$  is a rank-one locally free (hence torsion-free and $\mu$-semistable, since every fibre of a Weierstrass fibration is integral by assumption) sheaf on $X_s$.  Hence, by \cite[Proposition 6.51]{FMNT}, the restriction $F|_s$ is $\Psi_s$-WIT$_1$ (and $\mu$-semistable) for all $s \in S$, and the semistability locus of $F$ is the entirety of $S$.  This completes the proof of the lemma.
\end{proof}

\begin{remark}\label{remark3}
  Using the same notation as in Lemma \ref{lemma16} and its proof, we can tensor every term in \eqref{eq4} with $L^\ast$ to see that $E \otimes L^\ast$ is isomorphic to the ideal sheaf $I_C$ of some subscheme $C \subseteq X$ of codimension at least 2, while $T \otimes L^\ast \cong \OO_C$.  Since $T$ is $\Psi$-WIT$_1$ by Lemma \ref{lemma16}, all its subsheaves are $\Psi$-WIT$_1$ as well.  Therefore, when $X$ is a threefold, $T$ cannot have any 0-dimensional subsheaves, i.e.\ $T$ is a pure 1-dimensional sheaf if nonzero, in which case $C$ would be a pure 1-dimensional closed subscheme of $X$.
  \end{remark}

\begin{lemma}\label{lemma17}
Let $a,b, \pi$ be as in Lemma \ref{lemma16}.  Let $E$ be a rank-one torsion-free sheaf on $X$ with $\mu (E)=0$.  Then $E$ satisfies the vanishing condition \eqref{eq1} if and only if the cokernel $T$ of the canonical injection $E \hookrightarrow E^{\ast\ast}$ is $\Psi$-WIT$_1$.
\end{lemma}

\begin{proof}
To begin with, suppose $E$ satisfies \eqref{eq1}.  Let $L := E^{\ast\ast}$.  By the same argument as in  Lemma \ref{lemma16}, we obtain that $L$ is $\Psi$-WIT$_1$.    Then $E$ itself is $\Psi$-WIT$_1$, since it is a subsheaf of $L$.  By Lemma \ref{lemma7}, we have $\widehat{E} \in \mathcal B^\circ_Y$.  Since we have an injection $\Psi^0(T)\hookrightarrow \widehat{E}$, we get $\Psi^0(T)=0$, and so $T$ is $\Psi$-WIT$_1$.

For the converse, suppose $T$ is $\Psi$-WIT$_1$.  We still have that $L$ is $\Psi$-WIT$_1$.  For any $A \in \mathcal B_X \cap W_{0,X}$, we have the exact sequence
\[
\Hom (A,T) \to \Ext^1 (A,E) \to \Ext^1 (A,L)
\]
from \eqref{eq4}.  Since $A$ is $\Psi$-WIT$_0$ and $T$ is $\Psi$-WIT$_1$, we have $\Hom (A,T)=0$.  On the other hand, $\Ext^1 (A,L)=0$ by Theorem \ref{main1}.  Hence $\Ext^1 (A,E)$ vanishes, and the lemma is proved.
\end{proof}

%
%

\begin{remark}\label{remark4}
Take any $E \in \mathcal D_X$ of rank one, and suppose $\dimension X = 2$.  The cokernel of $E \hookrightarrow E^{\ast \ast}$ is 0-dimensional, and so must be zero since it is $\Psi$-WIT$_1$ by Lemma \ref{lemma17}.  Hence $E$ is locally free. The equivalence of categories in Proposition \ref{prop3} thus reduces to the last equivalence in \cite[Proposition 5.7]{Lo5}.
\end{remark}

\begin{lemma}\label{lemma20}
Let $T$ be a $\Psi$-WIT$_1$ coherent sheaf of codmension at least 2 on $X$.  Then $\dimension (\pi_\ast T) = \dimension (T) -1$, i.e.\ for a general closed point $s \in \text{supp}(\pi_\ast T)$, the restriction $T|_s$ is 1-dimensional.
\end{lemma}

\begin{proof}
Since $\pi$ is a fibration of relative dimension 1, it suffices to show $\dimension (\pi_\ast T) \leq \dimension (T) - 1$.  Suppose $\dimension (\pi_\ast T) = \dimension (T) = n-2$.  Then for a general closed point $s \in S_1 := \text{supp}(\pi_\ast T)$, the restriction $T|_s$ is 0-dimensional.  Let $\iota$ denote the closed immersion $S_1 \hookrightarrow S$, and
 \[
  \iota_X : X_{S_1} \hookrightarrow X, \text{\quad} \iota_Y : Y_{S_1} \hookrightarrow Y
  \]
  be the corresponding closed immersions  obtained after base change.  We have $T = {\iota_X}_\ast \tilde{T}$ for some coherent sheaf $\tilde{T}$ on $X_{S_1}$.  Since
\[
\widehat{T}[-1] \cong \Psi (T)  =   \Psi ({\iota_X}_\ast \tilde{T}) \cong {\iota_Y}_\ast (\Psi_{S_1} (\tilde{T}))
\]
by base change (see \cite[(6.3) and Proposition A.85]{FMNT}), we see that $\tilde{T}$ itself is $\Psi_{S_1}$-WIT$_1$.  By \cite[Corollary 6.3]{FMNT}, for any closed point $s \in S_1$ we have
\begin{equation}\label{eq22}
(  \Psi_{S_1}^1 (\tilde{T}) )|_s \cong \Psi_s^1 (\tilde{T}|_s).
\end{equation}
For a general closed point $s \in S_1$, however, the restriction $\tilde{T}|_s$ is a 0-dimensional sheaf (since $\dimension \tilde{T} = \dimension T = \dimension (\pi_\ast T)$ by assumption), which is $\Psi_s$-WIT$_0$.  Hence the right-hand side of \eqref{eq22} vanishes for a general $s \in S_1$, and so the left-hand side of \eqref{eq22} also vanishes for a general $s \in S_1$.  Since $\tilde{T}$ is $\Psi_{S_1}$-WIT$_1$, this implies that $\tilde{T}|_s$ vanishes for a general $s \in S_1$, contradicting our assumption that $S_1$ is the support of $\pi_\ast T$.  Therefore, it must be the case that $\dimension (\pi_\ast T) \leq \dimension (T) -1$.
\end{proof}


\begin{lemma}\label{lemma21}
Let $a,b, \pi$ be as in Lemma \ref{lemma16}.  Any rank-one object in $\mathcal D_X$ is a locally free sheaf.
\end{lemma}

\begin{proof}
Let $E$ be any rank-one object in $\mathcal D_X$, and let $T$ be as in \eqref{eq4}.  The argument in the proof of Lemma \ref{lemma16} shows that all the terms in \eqref{eq4} are $\Psi$-WIT$_1$, and that $E^{\ast\ast} \in \mathcal D_X$.  Thus  we obtain a short exact sequence in $\Coh (Y)$
\begin{equation}\label{eq23}
0 \to \widehat{E} \to \widehat{E^{\ast\ast}} \to \widehat{T} \to 0
\end{equation}
in which all the terms are $\Phi$-WIT$_0$.

By Proposition \ref{prop3} and Lemma \ref{lemma15}, we know that we have a section $\theta : S \to \text{supp}(\widehat{E^{\ast\ast}})$ of $\widehat{\pi}$.  And, if we write $\kappa$ for the closed immersion $\theta (S) \hookrightarrow Y$, then $\widehat{E^{\ast\ast}} \cong \kappa_\ast A$ for some line bundle $A$ on $\theta (S)$.  Applying the same argument to $\widehat{E}$, we see that $\text{supp}(\widehat{E}) = \text{supp}(\widehat{E^{\ast\ast}})$, and $\widehat{E} \cong \kappa_\ast A'$ for some line bundle $A'$ on $\theta (S)$.

Now, by Lemma \ref{lemma20}, we have $\dimension (\pi_\ast T) \leq n-3$.  Since $\widehat{E^{\ast\ast}}$ is supported on a section of $\widehat{\pi}$, and the support of $\widehat{T}$ is contained in the support of $\widehat{E^{\ast\ast}}$, for any closed point $s \in \text{supp}(\pi_\ast T)$, the restriction $\widehat{T}|_s$ must be 0-dimensional.  Hence $\dimension \widehat{T} \leq n-3$.    On the other hand, we can write $\widehat{T} = \kappa_\ast T'$ for some coherent sheaf $T'$ on $\text{supp}(\widehat{E^{\ast\ast}})$.  Then $T'$ has codimension at least 2 as a coherent sheaf on $\theta (S)$, and we have a short exact sequence
\[
  0 \to \kappa_\ast A' \to \kappa_\ast A \to \kappa_\ast T' \to 0 \text{\quad in $\Coh (Y)$},
\]
which gives a short exact sequence
\begin{equation}\label{eq25}
  0 \to A' \to A \to T' \to 0 \text{\quad in $\Coh (\theta (S))$}.
\end{equation}
However, in the short exact sequence \eqref{eq25}, both $A'$ and $A$ are reflexive sheaves, while $T'$ has codimension at least 2 on $\theta (S)$, contradicting \cite[Corollary 1.5]{SRS} if $\widehat{T}$ is nonzero.
Hence $\widehat{T}$ must vanish, i.e.\ $E$ itself is a line bundle on $X$.
\end{proof}

Putting the above results together, we obtain:

\begin{theorem}\label{theorem6}
Let $a,b, \pi$ be as in Lemma \ref{lemma16}.  The equivalence \eqref{eq20} in Proposition \ref{prop3} restricts to an equivalence
\begin{multline}\label{eq24}
\{ \text{line bundles of fibre degree 0 on $X$}\}  \\
      \leftrightarrow \{ F \in \Coh (Y) : F \text{ is pure of codimension 1, flat over $S$, $d(F)=1$}\} \\
      =  \{ \tau_
      \ast L : \tau \text{ is a section of }\widehat{\pi}, L \in \text{Pic}(S)\}.
\end{multline}
\end{theorem}

\begin{proof}
First, we show that the rank-one objects in the category $\mathcal D_X$ are exactly the line bundles of fibre degree 0 on $X$.  That any rank-one object in $\mathcal D_X$ is a line bundle of fibre degree 0 follows from Lemma \ref{lemma21}.  That any line bundle of fibre degree 0 lies in $\mathcal D_X$ follows from \cite[Proposition 6.51]{FMNT} and Theorem \ref{main1}.

That the second the third categories above are equivalent follow from Lemma \ref{lemma15}.
\end{proof}

\appendix

\section{Polynomial Stability Conditions} \label{PSC} 

Polynomial stability was defined on $D^b(X)$ by Bayer for any normal projective variety $X$ \cite[Theorem 3.2.2]{BayerPBSC}.  While the central charge for a Bridgeland stability condition takes values in $\Cfield$, the central charge for a polynomial stability condition takes values in the abelian group  $\Cfield [m]$ of polynomials over $\Cfield$.

The polynomial stability conditions we  concern ourselves with in this paper consist of the following data, where $X$ is a smooth projective $n$-fold:
\begin{enumerate}
\item the heart $\Ap = \langle \Coh_{\leq n-2}(X), \Coh_{\geq n-1}(X)[1]\rangle$, and
\item a group homomorphism (the central charge) $Z : K(X) \to \Cfield [m]$ of the form
$$ Z(E)(m) = \sum_{d=0}^n \int_X  \rho_d H^d  \cdot ch(E) \cdot U \cdot m^d$$
where
\begin{enumerate}
\item the $\rho_d \in \Cfield$ are nonzero, satisfy $\rho_0, \cdots, \rho_{n-2} \in \mathbb{H}$, $\rho_{n-1}, \rho_n \in -\mathbb{H}$, and their configurations are of either type W1 or W2 as defined in the beginning of Section \ref{moduli1},
\item $H \in \text{Amp}(X)_\RR$ is an ample class, and
\item $U =1+N$ where $N \in A^\ast(X)_{\mathbb{R}}$ is concentrated in positive degrees.
\end{enumerate}
\end{enumerate}

The configuration of the $\rho_i$ is compatible with the heart $\Ap$, in the sense that for every nonzero $E \in \Ap$, we have $Z(E)(m) \in \mathbb{H}$ for $m \gg 0$.  So there is a uniquely determined function germ $\phi (E)$  such that
\[
 Z(E)(m) \in \mathbb{R}_{>0} e^{i \pi \phi (E)(m)} \text{ for all } m \gg 0.
 \]
This allows us to define the notion of semistability for objects in $\Ap$.  We say that a nonzero object $E \in \Ap$ is \textit{$Z$-semistable} (resp.\ \textit{$Z$-stable}) if, for any nonzero subobject $G \hookrightarrow E$ in $\Ap$, we have $\phi (G)(m) \leq \phi (E)(m)$ for $m \gg 0$ (resp.\ $\phi (G)(m) < \phi (E)(m)$ for $m \gg 0$).  We also write $\phi (G) \preceq \phi (E)$ (resp.\ $\phi (G) \prec \phi (E)$) to denote this.  Harder-Narasimhan filtrations for polynomial stabilities exist \cite[Section 7]{BayerPBSC}.  The reader may refer to \cite{BayerPBSC} for more on the basics of polynomial stability.

\end{document}